\documentclass[reqno, 11pt, letterpaper ]{amsart}

\usepackage{amsmath}
\usepackage{amsfonts}
\usepackage{amssymb}
\usepackage{graphicx}
\usepackage{amsthm,color,yfonts,cite} \usepackage{paralist}
\usepackage{hyperref}
\usepackage{physics}
\usepackage{dsfont}

\newtheorem{theorem}{Theorem}

\newtheorem{proposition}[theorem]{Proposition}
\newtheorem{lemma}[theorem]{Lemma}

\theoremstyle{remark}
\newtheorem{remark}[theorem]{Remark}

\makeatletter
\newcommand*{\rom}[1]{\expandafter\@slowromancap\romannumeral #1@}
\makeatother

\newcommand{\ls}{\lesssim}

\newcommand{\la}{\langle}
\newcommand{\ra}{\rangle}
\newcommand{\R}{\mathbb{R}}

\newcommand{\Z}{\mathbb{Z}}
\newcommand{\pa}{\partial}
\newcommand{\ep}{\epsilon}
\newcommand{\al}{\alpha}

\usepackage{wrapfig}
\usepackage{tikz}
\usetikzlibrary{arrows,calc,decorations.pathreplacing}
\definecolor{light-gray1}{gray}{0.90}
\definecolor{light-gray2}{gray}{0.80}
\definecolor{light-gray3}{gray}{0.60}

\numberwithin{equation}{section}

\numberwithin{theorem}{section}

\numberwithin{table}{section}

\numberwithin{figure}{section}

\ifx\pdfoutput\undefined
  \DeclareGraphicsExtensions{.pstex, .eps}
\else
  \ifx\pdfoutput\relax
    \DeclareGraphicsExtensions{.pstex, .eps}
  \else
    \ifnum\pdfoutput>0
      \DeclareGraphicsExtensions{.pdf}
    \else
      \DeclareGraphicsExtensions{.pstex, .eps}
    \fi
  \fi
\fi

\title[BBM regularization of KdV]{On the Benjamin-Bona-Mahony regularization of the Korteweg-de Vries equation}

\date{\today}
\author[Y. Hong]{Younghun Hong}
\address{Department of Mathematics, Chung-Ang University, Seoul 06974, South Korea}
\email{yhhong@cau.ac.kr}

\author[J. Jang]{Junyeong Jang}
\address{Department of Mathematics, Chung-Ang University, Seoul 06974, South Korea}
\email{jyjang0119@cau.ac.kr}

 \author[C. Yang]{Changhun Yang}
\address{Department of Mathematics, Chungbuk National University, Cheongju-si 28644, Chungcheongbuk-do, Korea}
\email{chyang@chungbuk.ac.kr}

\begin{document}

\begin{abstract}
The Benjamin-Bona-Mahony equation (BBM) is introduced as a regularization of the Korteweg-de Vries equation (KdV) for long water waves [T.B. Benjamin, J.L. Bona, and J.J. Mahony, Philos. Trans. Roy. Soc. London Ser. A 272(1220) (1972), pp. 47–78]. In this paper, we establish the convergence from the BBM to the KdV for energy class solutions. As a consequence, employing the conservation laws, we extend the known temporal interval of validity for the BBM regularization.
\end{abstract}

\maketitle

\section{Introduction}

\subsection{Background}

The Benjamin-Bona-Mahony equation (BBM)
\begin{equation}\label{eq: BBM0}
\partial_t (1-\al\partial_x^2)u+\partial_x u+\al \partial_x(u^2)=0
\end{equation}
is a model equation describing the uni-directional propagation of long water waves, where $u=u(t,x): I(\subset\mathbb{R})\times\mathbb{R}\to\mathbb{R}$ is the amplitude of a water surface from an equilibrium and $\alpha>0$ represents $\frac{\textup{(amplitude)}}{\textup{(depth)}}$ for water waves\footnote{In the literature, $\frac{\textup{(amplitude)}}{\textup{(depth)}}$ is typically denoted by $\epsilon$. However, in this article, a different Greek letter $\alpha$ is assigned to this physical quantity, and let $\epsilon=\sqrt{\alpha}$ for notational convenience (see Remark \ref{remark: reformulation of BBM} $(ii)$).}. This equation has been introduced in the original work of Benjamin, Bona and Mahony \cite{BBM1972} as an alternative to the Korteweg-de Vries equation (KdV) in the theory of water waves. Indeed, it is known that for long waves, the KdV equation
\begin{equation}\label{eq: KdV0}
\partial_t u+\partial_x u+\al\partial_x^3 u+\alpha \partial_x(u^2)=0
\end{equation}
gives a formal approximation to the water wave problem in the sense that
\begin{equation}\label{eq: water wave problem}
\partial_t u+\partial_x u+\al\partial_x^3 u+\alpha \partial_x(u^2)=O(\alpha^2),
\end{equation}
provided that the parameter $\al>0$ is small enough (see \cite{Cra85, KN86, SW00, SW02, BCL2005, Wri2005, Dul12} for rigorous justification. We also refer to Lannes \cite[Chapter 7]{Lannes2013} and Schneider-Uecker \cite[Chapter 12]{SU2017} for more details). However, if $u(t,x)$ is regular, collecting all small terms in $O(\alpha)$, it can be written as 
$$\partial_t u+\partial_x u=O(\al).$$
Then, differentiating in $x$ twice and substituting by $\partial_x^3 u=-\partial_t\partial_x^2u+O(\al)$, \eqref{eq: water wave problem} becomes
$$\partial_t (1-\al\partial_x^2)u+\partial_x u+\al \partial_x(u^2)=O(\al^2).$$
Subsequently, dropping the small $O(\alpha^2)$-term, the BBM \eqref{eq: BBM0} is obtained. This formal analysis shows that the BBM \eqref{eq: BBM0} can be considered as another approximation of the same order to the original water wave problem as the KdV.

One of notable features of the BBM is that, compared to the KdV, the derivative nonlinearity $\partial_x(u^2)$ is regularized by inverting $(1-\al\partial_x^2)$ in \eqref{eq: BBM0}, so it is easier to analyze. For this reason, the well-posedness of the BBM is established \cite{BBM1972, BT2009, Rou2010,Pan2011} earlier than that of the KdV \cite{Bourgain1993_2, KPV1991, KPV1993, KPV1996, CKSTT2003, KT2006, KV2019}, and its proof is much simpler. Another important aspect is that as will be described below, the BBM also approximates the KdV in the long-wave limit \cite{BBM1972, BPS1983} (see Remark \ref{remark: long-wave reformulation}). Therefore, in conclusion, the BBM can be used as a \textit{regularized} alternative to the KdV for the water wave problem.

This regularization method, called the \textit{BBM regularization} or the \textit{BBM trick}, is known to be robust, and has been employed for many other equations. In a similar context, the Boussinesq system is regularized by the BBM \cite{AABCW2006}, and the ill-posed \textit{bad} Boussinesq system is formally converted into the \textit{regularized(or improved)} one \cite{Whi1974, PW1997, BCS2002}. Another application of the BBM trick can be found in the derivation of nonlinear Schrödinger equations with improved dispersion from deep water models \cite[Section 8]{Lannes2013} or in the context of nonlinear optics \cite{CL2009, Lan2011}.

\subsection{Main results}

In this paper, we revisit the proof of the BBM regularization of the KdV in the classical work of Bona, Pritchard and Scott \cite[Theorem 1]{BPS1983}, but we try to investigate the problem using the Fourier analysis methods developed in analysis of dispersive equations.

To set up the problem, inserting $u(t,x)=u_\epsilon(\epsilon^2 t, x-t)$ with $\ep=\sqrt{\al}$ in the BBM \eqref{eq: BBM0} and inverting $(1-\epsilon^2\partial_x^2)$, we deduce the rescaled BBM equation ($\textup{BBM}_\epsilon$)
\begin{equation}\label{eq: BBM}
\partial_tu_\epsilon+\frac{\partial_x^3}{1-\epsilon^2\partial_x^2}u_\epsilon+\frac{\partial_x}{1-\epsilon^2\partial_x^2}(u_\epsilon^2)=0,
\end{equation}
where $u_\epsilon(t,x): I(\subset\mathbb{R})\times\mathbb{R}\to\mathbb{R}$ and $\frac{1}{1-\epsilon^2\partial_x^2}$ is the Fourier multiplier of symbol $\frac{1}{1+\epsilon^2\xi^2}$. By the series expansion $\frac{1}{1-\epsilon^2\partial_x^2}=1+\epsilon^2\partial_x^2+\epsilon^4\partial_x^4+\epsilon^6\partial_x^6+\cdots$, the normalized KdV
\begin{equation}\label{eq: KdV}
\partial_tw+\partial_x^3w+\partial_x(w^2)=0
\end{equation}
is formally derived in the limit $\epsilon\to 0$. In this article, we are concerned with the convergence from the $\textup{BBM}_\epsilon$ to the KdV \eqref{eq: KdV}.

\begin{remark}\label{remark: reformulation of BBM}
$(i)$ For the BBM regularization, the convergence from \eqref{eq: BBM} to \eqref{eq: KdV} is less confusing than its equivalent formulation from \eqref{eq: BBM0} to \eqref{eq: KdV0}, since the coefficients of the target equation \eqref{eq: KdV} are independent of the parameter for which the limit will be taken.\\
$(ii)$ For notational convenience, we denote the physical quantity $\frac{\textup{(amplitude)}}{\textup{(depth)}}$ by $\epsilon^2$ instead of $\epsilon$, since the natural frequency scaling for \eqref{eq: BBM} is $\xi\sim\frac{1}{\epsilon}$.
\end{remark}

The $\textup{BBM}_\epsilon$ is globally well-posed in $H^s$ for $s\geq 0$ \cite{BBM1972, BT2009}, and so is the KdV in $H^s$ for $s\geq -1$ \cite{KV2019}. Moreover, the $\textup{BBM}_\epsilon$ has three conservation laws \cite{Olver1979}, and that the first three conserved quantities for the KdV among infinitely many ones correspond to the conserved quantities for the $\textup{BBM}_\epsilon$;
\begin{equation}\label{eq: conservation laws}
\begin{aligned}
&E_{\textup{BBM}_\epsilon}^{(0)}[u_\epsilon]:=\int_{\mathbb{R}} u_\epsilon dx &&\leftrightarrow && E_{\textup{KdV}}^{(0)}[w]:=\int_{\mathbb{R}} wdx,\\
&E_{\textup{BBM}_\epsilon}^{(1)}[u_\epsilon]:=\int_{\mathbb{R}} u_\epsilon^2+\epsilon^2(\partial_x u_\epsilon)^2dx&&\leftrightarrow && E_{\textup{KdV}}^{(1)}[w]:=\int_{\mathbb{R}} w^2dx,\\
&E_{\textup{BBM}_\epsilon}^{(2)}[u_\epsilon]:=\int_{\mathbb{R}} \frac{1}{2}(\partial_xu_\epsilon)^2-\frac{1}{3}u_\epsilon^3dx&&\leftrightarrow && E_{\textup{KdV}}^{(2)}[w]:=\int_{\mathbb{R}} \frac{1}{2}(\partial_xw)^2-\frac{1}{3}w^3dx
\end{aligned}
\end{equation}
(see Appendix \ref{sec: proof of conservation laws for the rescaled BBM equation}).

Our first main result asserts that the BBM regularization works for energy-class solutions.

\begin{theorem}[Local-in-time BBM regularization of the KdV]\label{thm: main result 1.1}
Suppose that $1\leq s\leq 5$ and 
\begin{equation}\label{eq: main result 1 initial data condition}
\sup_{\ep\in(0,1]}\|u_{\ep,0}\|_{H^s}, \|w_0\|_{H^s}\leq R,
\end{equation}
and let $u_\ep(t)\in C_t(\mathbb{R};H_x^s)$ (resp., $w(t)\in C_t(\R;H_x^s)$) be the unique solution to the $\textup{BBM}_\epsilon$ \eqref{eq: BBM} (resp., the KdV \eqref{eq: KdV}) with the initial data $u_{\epsilon, 0}$ (resp., $w_0$). Then, there exists $T=T(R)>0$, independent of $\epsilon\in(0,1]$, such that 
$$\|u_\ep(t)-w(t)\|_{C_t([-T,T];L_x^2)}\lesssim_R \|u_{\ep,0}-w_0\|_{L^2}+\ep^{\frac{2s}{5}}.$$
\end{theorem}

\begin{remark}
$(i)$ Compared to the classical result in \cite[Theorem 1]{BPS1983}, the regularity requirement for the BBM regularization is reduced down to $s=1$ from $s=5$.\\
$(ii)$ In Theorem \ref{thm: main result 1.1}, we do not seek to find the optimal regularity condition, because it is enough for our second main result. The condition $s\geq 1$ is expected to be relaxed in some extent, but it seems difficult to include the case $s<1$ due to the complicated structure of the symbol of the rescaled linear BBM flow (see Remark \ref{rmk:tech difficulties}).\\
$(iii)$ The $O(\ep^{\frac{2s}{5}})$ rate of convergence is from regularity gap and the difference between the two symbols for the linear propagators associated with the rescaled BBM and Boussinesq equations (see \eqref{ineq : v-w 1.1} for details).
\end{remark}

It is important to note that including a larger class of solutions for the BBM regularization is not a purely mathematical question. As a direct consequence, the temporal interval of valid approximation can be extended by the conservation laws, which is the second main result of the paper.

\begin{theorem}[Global-in-time BBM regularization for the KdV]\label{thm: main result 1.2}
Under the assumptions and the notations in Theorem \ref{thm: main result 1.1} with $s=1$ and $R>0$, there exists $K=K(R)>0$ such that for all $t\in\mathbb{R}$,
$$\|u_\ep(t)-w(t)\|_{L_x^2}\lesssim_R \big(\|u_{\ep,0}-w_0\|_{L_x^2}+\ep^{\frac{2}{5}}\big)e^{K|t|},$$
where the implicit constant is independent of $\epsilon\in(0,1]$ and $t\in\mathbb{R}$.
\end{theorem}

\begin{remark}
Theorem \ref{thm: main result 1.2} shows that for the rescaled model \eqref{eq: BBM}, the BBM regularization is valid on the time interval as long as
$$|t|\ll \frac{1}{K}\ln\bigg(\frac{1}{\|u_{\ep,0}-w_0\|_{L_x^2}+\ep^{2/5}}\bigg).$$
Thus, the interval of valid approximation increases logarithmically as $\epsilon\to0$. 
\end{remark}

\begin{remark}[Reformulation of Theorem \ref{thm: main result 1.1} and \ref{thm: main result 1.2}]\label{remark: long-wave reformulation}
To compare with the previous result, we rephrase the two main theorems in the setup of \cite[Theorem 1]{BPS1983} as follows. For $u_0\in H^s$ with $1\leq s\leq 5$, let $u(t)\in C(\mathbb{R}; H^s)$ be the global solution to the normalized BBM equation
\begin{equation}\label{eq: BBM normalized}
\partial_tu+\partial_x u +\partial_x(u^2)-\partial_t\partial_x^2 u=0
\end{equation}
with initial data $\al u_0(\al^{\frac{1}{2}}x)$, and let $w(t)\in C(\mathbb{R}; H^s)$ be the global solution to the KdV equation \eqref{eq: KdV} with the initial data $u_0$. Then, by scaling, Theorem \ref{thm: main result 1.1} implies that there exists $T>0$ such that 
\begin{equation}\label{eq: main result 1.1, rephrased}
\sup_{|t|\leq \frac{T}{\al^{3/2}}}\|u(t)-\al w(\al^\frac{3}{2}t, \al^{\frac{1}{2}}(\cdot-t))\|_{L_x^2}\lesssim_R \al^{\frac{s}{5}+\frac{3}{4}}
\end{equation}
and if $s=1$, then 
\begin{equation}\label{eq: main result 1.2, rephrased}
\|u(t)-\al w(\al^\frac{3}{2}t, \al^{\frac{1}{2}}(\cdot-t))\|_{L_x^2}\lesssim_R \al^{\frac{19}{20}}e^{K \al^{\frac{3}{2}}|t|}.
\end{equation}
When $s=5$, \eqref{eq: main result 1.1, rephrased} recovers \cite[Theorem 1]{BPS1983}. When $s=1$, \eqref{eq: main result 1.2, rephrased} shows that $u(t)\approx\al w(\al^\frac{3}{2}t, \al^{\frac{1}{2}}(\cdot-t))$ on the time interval of size $\sim \al^{-\frac{3}{2}}\ln(\frac{1}{\al})$. 
\end{remark}

\subsection{Ideas of the proof}

The proof of the main results follows the general approach in the recent work of the first and the third authors and their collaborators  \cite{HY2019, HKY2021, HKYN2021, HKY2022, HKY2023, HY2024}, where the Fourier analysis method has been successfully applied to convergence/approximation problems. In particular, it follows closely from Hong-Yang \cite{HY2024}, where the convergence from the Boussinesq equation to the KdV has been established in a similar context.  However, several new ideas are introduced to adjust the strategy to better suit the BBM equation as well as to extend the temporal interval of validity.

First, a simple but important remark is that for lower regularity solutions, it is necessary to make use of the equations in the integral form, not in the differential form. Thus, we consider the integral representation of the rescaled BBM \eqref{eq: BBM}; 
\begin{equation}\label{integral BBM}
u_\ep(t,x)=S_\ep(t)u_{\ep,0}-\int_0^tS_\ep(t-t_1)\frac{\partial_x}{1-\ep^2\partial_x^2}\big(u_\ep(t_1)^2\big) dt_1,
\end{equation}
where 
\begin{equation}\label{linear BBM flow}
S_\epsilon(t):=e^{-\frac{t\partial_x^3}{1-\epsilon^2\partial_x^2}}=e^{its_\epsilon(-i\partial_x)}
\end{equation}
and $s_\epsilon(-i\partial_x)$ is the Fourier multiplier with symbol
\begin{equation}\label{BBM phase function}
s_\epsilon(\xi):=\frac{\xi^3}{1+\epsilon^2\xi^2}.
\end{equation}
Note that this equation formally converges to the Duhamel formula for the KdV \eqref{eq: KdV};
\begin{equation}\label{integral KdV}
w(t)=S(t)w_0-\int_0^tS(t-t_1)\partial_x\big(w(t_1)^2\big)dt_1,
\end{equation}
where
$$S(t):=e^{-t\partial_x^3}=e^{its(-i\partial_x)}$$
is the Airy flow and $s(\xi)=\xi^3$, because $s_\epsilon(\xi)=\frac{\xi^3}{1+\epsilon^2\xi^2}\to s(\xi)=\xi^3$ as $\ep\to 0$ for all $\xi\in\R$.

To analyze the equation \eqref{integral BBM}, we employ the Fourier restriction norm, known as the Bourgain space norm, introduced by Bourgain \cite{Bourgain1993, Bourgain1993_2}, because it is a natural choice of norm to deal with the nonlinear term $\frac{\partial_x}{1-\ep^2\partial_x^2}(u_\ep)^2$. Indeed, in order to include lower regularity solutions for the BBM regularization, it is important to obtain a uniform bound for the solution $u_\epsilon(t)$ to the BBM \eqref{integral BBM} in a lower regularity setting. Note that for fixed $\epsilon>0$, the nonlinear term $\frac{\partial_x}{1-\ep^2\partial_x^2}(u_\ep)^2$ in \eqref{integral BBM} is easy to handle because of the smoothing from $\frac{\partial_x}{1-\ep^2\partial_x^2}$. Nevertheless, in our case, it should be considered as an \textit{almost} derivative nonlinearity to be overcome, since the regularization effect gets weaker as $\frac{1}{1-\ep^2\partial_x^2}\to 1$ in the limit $\epsilon\to 0$. Therefore, we need to capture the local smoothing from the linear propagator $S_\epsilon(t)$ (see \eqref{linear BBM flow}). This observation leads us to employ the Fourier restriction norm, because it is known as a strong and convenient tool to deal with derivative nonlinearities for dispersive equations. For instance, for the KdV equation, the local well-posedness is proved in $H^s$ for $s>\frac{3}{2}$ by the energy method (see \cite{Kato1979}). The regularity requirement is reduced to $s>\frac{3}{4}$ using the Strichartz estimates, the local smoothing estimate and the maximal function estimate \cite{KPV1991}, and then reduced further to $s>-\frac{3}{4}$ using the bilinear estimate in the Fourier restriction norm \cite{KPV1996}. In our case $s\geq1$, the former approach may work, but we employ the latter one based on the Fourier restriction norm, since it is simpler.

In the proof of the BBM regularization (Theorem \ref{thm: main result 1.1}), a crucial step is to show the uniform-in-$\epsilon$ Fourier restriction norm bound for the $\textup{BBM}_\epsilon$ (Proposition \ref{Prop : BBM LWP}). As long as such uniform bounds are proved, using the properties of the Fourier restriction norm, one can estimate the difference between the two integral equations \eqref{integral BBM} and \eqref{integral KdV}. It is also important to note that the desired uniform bounds cannot be obtained directly from earlier well-posedness theorems in Benjamin-Bona-Mahony \cite{BBM1972} and Bona-Tzvetkov \cite{BT2009}. Precisely, one can see that if one normalizes the rescaled model \eqref{eq: BBM}, apply the known well-posedness result for the normalized model with $\epsilon=1$, and rescale back to \eqref{eq: BBM}, then the interval $[-T_\ep,T_\ep]$ of existence in one contraction mapping argument step depends on $\epsilon$, but $T_\epsilon$ also shrinks to zero as $\epsilon\to0$.

By the standard perturbative argument, the desired uniform bound for nonlinear solutions (Proposition \ref{Prop : BBM LWP}) can be obtained from the uniform bilinear estimates for the rescaled linear BBM flow (Proposition \ref{Lem:bilinear for BBM} and \ref{Lem:BL for difference}) in the Fourier restriction norm. Indeed, a major portion of this article is devoted to proving such bilinear estimates. This part is technical but important. Some new ideas are needed to deal with the symbol $s_\epsilon(\xi)$ (see \eqref{BBM phase function}). It turns out that the symbol $s_\epsilon(\xi)$ behaves in a complicated way in high frequencies (see Remark \ref{rmk:tech difficulties}) unlike the analogous symbol in the Boussinesq case \cite{HY2024}. To overcome the technical issue, we decompose the bilinear interaction according to frequency interactions. For the high-high and the high-low frequency interactions (Lemma \ref{Lem:high-high BL} and \ref{Lem:high-low BL}), we employ Strichartz estimates (Lemma \ref{Lem:Strichartz}), while the low-low frequency interaction is treated as usual (Lemma \ref{Lem:low-low BL}). Then, collecting all, we prove the bilinear estimates.

As for the second main result (Theorem \ref{thm: main result 1.2}), the proof of the global-in-time convergence is based on the simple observation that the conservation laws in \eqref{eq: conservation laws} are good enough to provide uniform-in-$\epsilon$ global-in-time $H^1$-norm upper bounds for energy class solutions to the rescaled BBM \eqref{eq: BBM} and the KdV \eqref{integral KdV}. As a consequence, the local-in-time result (Theorem \ref{thm: main result 1.1}) can be iterated arbitrarily many times. We emphasize that our proof relies heavily on the conservation laws for the equation \eqref{eq: BBM}, which come from the special algebraic property of the model equation and the choice of scaling. For instance, as for the KdV limit for the Boussinesq equation in \cite{HY2024}, the rescaled Boussinesq equation obeys conservation laws controling the $H^1$-norm of solutions only in the form of $\|\sqrt{1-\epsilon^2\partial_x^2} u_\epsilon(t)\|_{L^2}^2$ that is asymptotically the $L^2$-norm. For this reason, the conservation laws could not employed for the Boussinesq model \cite[Remark 1.2]{HY2024}.

\subsection{Organization of the paper}
The rest of the paper is organized as follows. In Section~\ref{sec: Prelim}, we provide the definition of the Fourier restriction norm and basic estimates. In Section~\ref{sec: Bilin est}, we prove bilinear estimates for the rescaled linear BBM flow in the Fourier restriction norm. Then, using the bilinear estimates, we establish local well-posedness and uniform bound for the $\textup{BBM}_\epsilon$ in Section~\ref{sec: uniform bounds for nonlinear solutions}. Finally, in Section~\ref{sec: proof of the main result}, we prove the main theorems. In addition, for readers' convenience, we provide a formal proof of the conservation laws for the $\textup{BBM}_\epsilon$ in Appendix~\ref{sec: proof of conservation laws for the rescaled BBM equation}. 

\subsection{Notations}\label{subsec: notations}
Throughout this article, $\ep\in (0,1]$ is a small parameter that will be sent to 0. In the theory of water waves, $\epsilon>0$ represents $\sqrt{\frac{\textup{(amplitude)}}{\textup{(depth)}}}$. For any $A, B>0$, we write
$$A\lesssim B\textup{ (resp., }A\gtrsim B, A\sim B)$$
if there exists $c>0$, independent of $\epsilon\in (0,1]$, such that $A\leq c B$ (resp., $A\geq c B$, $\frac{1}{c}A\leq B\leq cA$). Let
$$\la x\ra:=\sqrt{1+|x|^2}$$
be the standard Japanese bracket. We denote the Fourier multiplier with symbol $|\xi|$ (resp., $\la\ep\xi\ra$) by $|\partial_x|$ (resp., $\la\ep\partial_x\ra$).

In a sequel, the Littlewood-Paley theory will be employed in several steps. For this, we chose a smooth cut-off $\eta\in C_c^\infty(\R)$ such that $0\leq\eta\leq 1$, $\eta(\xi)=1$ if $\frac{6}{5}\leq|\xi|\leq \frac{9}{5}$, $\eta$ is supported in $[-2,2]\setminus [-\frac{1}{2},\frac{1}{2}]$, and $\sum_{N\in 2^{\mathbb{Z}}}\eta(\frac{\xi}{N})\equiv1$. Consequently, we define the Littlewood-Paley projection operator $P_{N}$ by
\begin{equation}\label{frequency truncation operator}
\widehat{P_{N}u} (\xi):=\eta\left(\frac{\xi}{N}\right)\hat{u}(\xi)
\end{equation}
and let $P_{\leq N}:=\sum_{M\leq N}P_M$.

\subsection{Acknowledgement}

The authors would like to thank Professor Jerry Bona for his interest in our work and for his kind explanation of the original idea of the BBM regularization during his visit in Seoul in the spring of 2024. This research was supported by the Chung-Ang University Graduate Research Scholarship in 2022. Y. Hong was supported by National Research Foundation of Korea (NRF) grant funded by the Korean government (MSIT) (No. RS-2023-00219980). C. Yang was supported by National Research Foundation of Korea (NRF) grant funded by the Korean government (MSIT) (No. 2021R1C1C1005700).

\section{Preliminaries}\label{sec: Prelim}

We briefly review the Fourier restriction norm, namely the Bourgain space norm, and its basic properties (see \cite[Chapter~7]{Linares2015} and \cite[Chapter 2]{Tao2006} for details). For $s,b\in\R$ and a characteristic hypersurface $p=p(\xi):\R\to\R$, the function space $X_{\tau=p(\xi)}^{s,b}$ is defined as the closure of the Schwartz class $\mathcal{S}(\R\times\R)$ under the Fourier restriction norm
$$\|u\|_{X_{\tau=p(\xi)}^{s,b}}:=\big\| \langle \xi \rangle^{s}\langle \tau -  p(\xi)\rangle^b \tilde{u}(\tau,\xi)\big\|_{L_{\tau,\xi}^2(\R\times \R)},$$
where $\tilde{u}$ is the space-time Fourier transform of $u$ given by
$$\tilde{u}(\tau,\xi)=\iint_{\mathbb{R}^2} u(t,x)e^{-i(t\tau+x\xi)} dxdt.$$
This function space is a Banach space having the nesting property $X^{s_2,b_2}_{\tau=p(\xi)}\subset X^{s_1,b_1}_{\tau=p(\xi)}$ for $s_1\leq s_2$ and $b_1\leq b_2$ and the duality relation $(X^{s,b}_{\tau=p(\xi)})^*=X^{-s,-b}_{\tau=-p(-\xi)}$. Moreover, it satisfies the following properties. 

\begin{lemma}[Basic properties of the Fourier restriction norm]\label{Lem:linear estimates of Xsb}
Let $\eta_T(t)=\eta(\frac{t}{T})\in C_c^\infty$ with $T\in(0,1]$. Then, the following hold for $s,b\in\R$.
\begin{enumerate}[$(1)$]
\item (Embedding) For $b>\frac12$, $X_{\tau=p(\xi)}^{s,b}\subset C_t(\mathbb{R}; H_x^s)$.
\item (Linear flow in $X_{\tau=p(\xi)}^{s,b}$) For $b>\frac{1}{2}$,
$$\|\eta_T(t)e^{itp(-i\partial_x)} u_0\|_{X_{\tau=p(\xi)}^{s,b}} \lesssim T^{\frac{1}{2}-b}\|u_0\|_{H^s_x}.$$
\item (Stability with respect to time localization) If $-\frac{1}{2}<b'\le b<\frac{1}{2}$, then 
$$\|\eta_T(t)u\|_{X_{\tau=p(\xi)}^{s,b'}} \lesssim T^{b-b'}\|u\|_{X_{\tau=p(\xi)}^{s,b}}.$$
If $\frac{1}{2}<b\leq1$, then
$$\|\eta_T(t)u\|_{X_{\tau=p(\xi)}^{s,b}} \lesssim T^{\frac{1}{2}-b}\|u\|_{X_{\tau=p(\xi)}^{s,b}}.$$
\item (Inhomogeneous term estimate) If $\frac12<b\leq 1$, then
$$\bigg\|\eta_T(t) \int_0^te^{i(t-t_1)p(-i\partial_x)}F(t_1)dt_1\bigg\|_{X_{\tau=p(\xi)}^{s,b}} \lesssim T^{\frac12-b}\|F\|_{X_{\tau=p(\xi)}^{s,b-1}}.$$
\item (Transference principle)
Let $b>\frac{1}{2}$. If the inequality
$$\|e^{it\tau_0}e^{itp(-i\partial_x)} u_0\|_{ L_t^q(\R;L_x^r)}\lesssim T^{\frac{1}{2}-b}\|u_0\|_{H^s_x}$$
holds for all $u_0\in H_x^s$ and $\tau_0\in\R$, then
$$\|u\|_{ L_t^q(\R;L_x^r)}\leq \|u\|_{X_{\tau=p(\xi)}^{s,b}}.$$
\end{enumerate}
\end{lemma}

In our setting, we employ the function space $X_\ep^{s,b}$ (resp. $X^{s,b}$) associated with the rescaled linear BBM flow $S_\ep(t)=e^{its_\epsilon(-i\partial_x)}$ (resp., the Airy flow $S(t)=e^{its(-i\partial_x)}$) with the norm
$$\|u\|_{X^{s,b}_{\ep}}:=\| \la\xi\ra^{s}\la\tau- s_\ep(\xi)\ra^b \tilde{u}\|_{L_{\tau,\xi}^2}\quad\left(\textup{resp., }\|u\|_{X^{s,b}}:=\|\la\xi \ra^{s}\la\tau- s(\xi)\ra^b \tilde{u}\|_{L_{\tau,\xi}^2}\right),$$
where $s_\epsilon(\xi)=\frac{\xi^3}{1+\epsilon^2\xi^2}$ and $s(\xi)=\xi^3$. In applications of the Fourier restriction norms, it is important to take the right choice, because the space-time Fourier transform of $e^{itp(-i\partial_x)} u_0$ is concentrated on the hypersurface $\tau=p(\xi)$ (see \cite[Chapter 2]{Tao2006}). Thus, the two norms $\|u\|_{X^{s,b}_{\ep}}$ and $\|u\|_{X^{s,b}}$ act completely differently in high frequencies in that $s_\epsilon(\xi)$ is asymptotically linear but $s(\xi)$ is cubic as $\xi\to\infty$. Nevertheless, they are comparable in low frequencies in the following sense.

\begin{lemma}\label{Lem:norm equiv in low freq}
Let $P_{\leq N_0}$ be the frequency truncation operator given by \eqref{frequency truncation operator} with a dyadic number $N_0\sim\frac{1}{2}\ep^{-\frac{2}{5}}$.
Then, for $0\leq\theta\leq 1$, the following norm equivalence holds :
$$\| P_{\le N_0} u \|_{{X^{s,b}}} \sim \| P_{\le N_0} u \|_{X_{\tau=\theta s_\epsilon(\xi)+(1-\theta)s(\xi))}^{s,b}}.$$
\end{lemma}

As an application, we can show the following low frequency approximation lemma. 

\begin{lemma}\label{Lem:difference linear sol in Xsb}
Let $\eta_T(t)=\eta(\frac{t}{T})\in C_c^\infty$, $T\in(0,1]$ and $P_{\leq N_0}$ be the frequency truncation operator given by \eqref{frequency truncation operator} with  a dyadic number $N_0\sim\frac{1}{2}\ep^{-\frac{2}{5}}$. Then, for $0\leq s\leq 5$ and $\frac{1}{2}<b\leq 1$, we have 
\begin{equation}\label{homogenous term}
\|\eta_T(t)(S_\ep(t)-S(t))P_{\leq N_0}u_0\|_{X^{0,b}} \lesssim \ep^{\frac{2s}{5}}T^{\frac{3}{2}-b} \|u_0\|_{H^s}
\end{equation}
and
\begin{equation}\label{Inhomogenous term}
\left\|\eta_T(t)\int_0^t (S_\ep(t-t_1)-S(t-t_1))\eta_T(t_1)(P_{\le N_0} F)(t_1)dt_1 \right\|_{X^{0,b}} \lesssim \ep^{\frac{2s}{5}}T^{\frac{3}{2}-b}\|F\|_{X^{s,b-1}}.
\end{equation}
\end{lemma}

For the proof of the above two lemmas, we refer to \cite[Lemma 3.3 and 3.4]{HY2024}. They will be employed later to compare the two equations \eqref{integral BBM} and \eqref{integral KdV}.

\section{Bilinear estimates for the rescaled linear BBM flow}\label{sec: Bilin est}

In this section, we establish the bilinear estimates in the Fourier restriction norm associated with the rescaled linear BBM flow (Proposition \ref{Lem:bilinear for BBM} and \ref{Lem:BL for difference}), which are the key ingredients in our analysis. 

\subsection{Bilinear estimate for the rescaled linear BBM flow}
The main result of this section is the following bilinear estimate.

\begin{proposition}[Bilinear estimate for the rescaled linear BBM flow]\label{Lem:bilinear for BBM}
For $s\geq1$ and sufficiently small $0<\delta\ll1$, we have
\begin{equation}\label{eq:bilinear for BBM}
\bigg\| \frac{\pa_x}{\la\ep\partial_x\ra^2}(uv)\bigg\|_{X_{\ep}^{s, -\frac{1-9\delta}{2}}}
\ls \|u\|_{X_{\ep}^{s,\frac{1+\delta}{2}}} \|v\|_{X_{\ep}^{s,\frac{1+\delta}{2}}}.
\end{equation}
\end{proposition}

\begin{remark}\label{cf. KdV flow}
$(i)$ Proposition~\ref{Lem:bilinear for BBM} is an analogue of the well-known bilinear estimate for the Airy flow;
\begin{equation}\label{eq: bilinear for KdV}
\|\partial_x(uv)\|_{X^{s,b-1}}\lesssim\|u\|_{X^{s,b}}\|v\|_{X^{s,b}}
\end{equation}
for $s>-\frac{3}{4}$ and some $\frac{1}{2}<b<1$. The inequality \eqref{eq: bilinear for KdV} is proved in the celebrated work of Kenig-Ponce-Vega \cite{KPV1996} (also refer to the book \cite{Linares2015}), and it is used to establish the local well-posedness of the KdV in the low regularity space  $H^s$ with $s>-\frac{3}{4}$. The inequality \eqref{eq: bilinear for KdV} is shown to be optimal in the sense that it fails when $s\leq-\frac{3}{4}$.\\
$(ii)$ For the bilinear estimate \eqref{eq:bilinear for BBM}, we restrict ourselves to the case $s\geq1$ even though it is not expected optimal, because it is sufficient to extend the interval of validity for the BBM approximation (Theorem \ref{thm: main result 1.2}). It might be desirable to reduce the regularity requirement, but we here do not pursue to do that, since the situation is more complicated below $s=1$ (see Remark \ref{rmk:tech difficulties} below). 
\end{remark}

\subsubsection{Reduction to the integral estimate}\label{sec: reduction}
Following the standard approach \cite{KPV1996}, the typical first step is to reduce the proof of the bilinear estimate to that of a certain integral estimate. Indeed, by the definition of the  norm and the Plancherel theorem, the inequality \eqref{eq:bilinear for BBM} is equivalent to 
$$\begin{aligned}
&\bigg\|\frac{\xi\la\xi\ra^s}{\la\ep\xi\ra^2\la\tau-s_\ep(\xi)\ra^{\frac{1-9\delta}{2}}}\iint_{\mathbb{R}^2}\frac{U(\tau_1, \xi_1)V(\tau-\tau_1,\xi-\xi_1)d\xi_1d\tau_1}{\la\xi_1\ra^s\la\xi-\xi_1\ra^s\la\tau_1-s_\ep(\xi_1)\ra^{\frac{1+\delta}{2}}\la\tau-\tau_1-s_\ep(\xi-\xi_1)\ra^{\frac{1+\delta}{2}}} \bigg\|_{L_{\tau,\xi}^2}\\
&\ls\|U\|_{L_{\tau,\xi}^2}\|V\|_{L_{\tau,\xi}^2}
\end{aligned}$$
taking $U(\tau, \xi)=\la\xi\ra^s\la\tau-s_\ep(\xi)\ra^{\frac{1+\delta}{2}}\tilde{u}(\tau,\xi)$ and $V(\tau,\xi)=\la\xi\ra^s\la\tau-s_\ep(\xi)\ra^{\frac{1+\delta}{2}}\tilde{v}(\tau,\xi)$. Hence, by the Cauchy-Schwarz inequality, it suffices to show that 
$$\frac{\xi^2\la\xi\ra^{2s}}{\la\ep\xi\ra^4\la\tau-s_\ep(\xi)\ra^{1-9\delta}}\iint_{\mathbb{R}^2}\frac{d\xi_1d\tau_1}{\la\xi_1\ra^{2s}\la\xi-\xi_1\ra^{2s}\la\tau_1-s_\ep(\xi_1)\ra^{1+\delta}\la\tau-\tau_1-s_\ep(\xi-\xi_1)\ra^{1+\delta}}$$
is bounded uniformly in $\tau,\xi\in\mathbb{R}$. However, in the above integral, one may drop $\frac{\la\xi\ra^{2s}}{\la\xi-\xi_1\ra^{2s}\la\xi_1\ra^{2s}}$, since $\la\xi\ra^{2s}\lesssim\la\xi_1\ra^{2s}\la\xi-\xi_1\ra^{2s}$ for $s\geq 0$. On the other hand, by the elementary inequality, we have
$$\int_{\mathbb{R}}\frac{d\tau_1}{\la\tau_1-s_\ep(\xi_1)\ra^{1+\delta}\la\tau-\tau_1-s_\ep(\xi-\xi_1)\ra^{1+\delta}}\lesssim \frac{1}{\langle \tau-s_\ep(\xi-\xi_1)-s_\ep(\xi_1)\rangle^{1+\delta}}.$$
Thus, it can be further reduced to show the uniform bound for the integral 
$$I_\ep(\tau,\xi):=\frac{\xi^2}{\la\ep\xi\ra^4\la\tau-s_\ep(\xi)\ra^{1-9\delta}}\int_\R\frac{d\xi_1}{\la\tau-s_\ep(\xi-\xi_1)-s_\ep(\xi_1)\ra^{1+\delta}}\lesssim 1.$$
Here, we may assume that $\xi>0$, because $I_\ep(\tau,-\xi)=I_\ep(-\tau,\xi)$. We also note that in $I_\epsilon(\tau,\xi)$, the integrand $\frac{1}{\la\tau-s_\ep(\xi-\xi_1)-s_\ep(\xi_1)\ra^{1+\delta}}$ is symmetric with respect to $\xi_1=\frac{\xi}{2}$. Therefore, we conclude that Proposition~\ref{Lem:bilinear for BBM} follows from the uniform bound 
\begin{equation}\label{eq: bilinear for BBM}
\sup_{\xi>0,\ \tau\in\mathbb{R}}\frac{2\xi^2}{\la\ep\xi\ra^4\la\tau-s_\ep(\xi)\ra^{1-9\delta}}\int_\frac{\xi}{2}^\infty\frac{d\xi_1}{\la\tau-s_\ep(\xi-\xi_1)-s_\ep(\xi_1)\ra^{1+\delta}}\lesssim1.
\end{equation}

\begin{remark}\label{rmk:tech difficulties}
In the standard approach, one might attempt to estimate $I_\epsilon(\tau,\xi)$ changing the variable by $z(\xi_1)=s_\ep(\xi_1)+s_\ep(\xi-\xi_1)$. However, it turns out that $z(\xi_1)$ is not one-to-one, and rather complicated case study is required, because the sign of
\begin{equation}\label{derivative of z}
\begin{aligned}
 z'(\xi_1)
 &= \frac{3\xi_1^2}{\langle \epsilon\xi_1\rangle^2}-\frac{2\ep^2\xi_1^4}{\langle \ep\xi_1\rangle^4}
-\Big( \frac{3(\xi-\xi_1)^2}{\langle \epsilon(\xi-\xi_1)\rangle^2}-\frac{2\ep^2(\xi-\xi_1)^4}{\langle \ep(\xi-\xi_1)\rangle^4} \Big) \\ 
&=\frac{\xi(2\xi_1-\xi)}{\langle\epsilon\xi_1\rangle^2\langle\epsilon(\xi-\xi_1)\rangle^2}\Big( \frac{2}{\langle\epsilon\xi_1\rangle^2}+\frac{2}{\langle\epsilon(\xi-\xi_1)\rangle^2} -1\Big) 
\end{aligned}
\end{equation}
may change multiple times depending on $\xi>0$.
\end{remark}

To overcome the technical difficulty, we modify the usual argument assuming more regularity. Precisely, we take out the high-high interaction from the bilinear term $\frac{\pa_x}{\la\ep\partial_x\ra^2}(uv)$, where the range of high frequencies is selected so that the sign of $z'(\xi_1)$ changes only in the high-high interaction. Then, the remaining part can be dealt with by the standard method (see Lemma \ref{Lem:low-low BL} and \ref{Lem:high-low BL}). On the other hand, for the high-high interaction, the uniform Strichartz estimates for the rescaled linear flow (Lemma \ref{Lem:Strichartz}) and the transference principle (Lemma \ref{Lem:linear estimates of Xsb} $(5)$) are employed to obtain the desired bilinear estimate (Lemma \ref{Lem:high-high BL}), but it requires to assume that $s\geq 1$. Then, collecting all, Proposition \ref{Lem:bilinear for BBM} follows.

\subsubsection{Proof of Proposition \ref{Lem:bilinear for BBM}}

First, we show the bilinear estimate for the high-high interaction. 
\begin{lemma}[Bilinear estimate; high-high interaction]\label{Lem:high-high BL}
For $s\geq 1$, we have
$$\left\| \frac{\pa_x}{\langle\epsilon\partial_x\rangle^2}\big((P_{>\frac{1}{5\epsilon}}u)(P_{>\frac{1}{5\epsilon}}v)\big)\right\|_{X_{\ep}^{s, -\frac{1-9\delta}{2}}}
\ls \|u\|_{X_{\ep}^{s,\frac{1+\delta}{2}}} \|v\|_{X_{\ep}^{s,\frac{1+\delta}{2}}},$$
provided that $u\equiv 0$ and $v\equiv 0$ outside the time interval $[-1,1]$.
\end{lemma}

For the proof, we employ the following Strichartz estimate. 

\begin{lemma}[Strichartz estimates]\label{Lem:Strichartz}
For any admissible pair $(q,r)$ such that $\frac{3}{q}+\frac{1}{r}=\frac{1}{2}$, $6<q\leq\infty$ and $2\leq r<\infty$, we have
\begin{equation}\label{Strichartz1}
    \big\|S_\epsilon(t)P_{>\frac{1}{5\epsilon}} u_0\big\|_{L_t^q(\mathbb{R}; L_x^r)}\lesssim \epsilon^{\frac{4}{q}}\big\||\partial_x|^{\frac{4}{q}}u_0\big\|_{L_x^2}.
\end{equation}
\end{lemma}

\begin{proof}

The Littlewood-Paley inequality and the Minkowski inequality leads to 
$$\big\|S_\epsilon(t)u_0\big\|_{L_t^q(\mathbb{R}; L_x^r)}\lesssim
\big\|\big(\sum_{N\in2^{\Z}}|P_N S_\ep(t)u_0|^2 \big)^{\frac{1}{2}} \big\|_{L_t^q(\mathbb{R}; L_x^r)}
\lesssim \big(\sum_{N\in2^{\Z}}\big\|P_N S_\ep(t)u_0\|_{L_t^q(\mathbb{R}; L_x^r)}^2 \big)^{\frac{1}{2}},$$
since $q, r \geq 2$. Hence, it suffices to show that for $N\in 2^{\mathbb{Z}}$ with $N>\frac{1}{5\epsilon}$, 
\begin{equation}\label{eq: dyadic piece Strichartz}
\big\|S_\epsilon(t)P_N u_0\big\|_{L_t^q(\mathbb{R}; L_x^r)}\lesssim (\epsilon N)^{\frac{4}{q}}\|u_0\|_{L_x^2}.
\end{equation}
Indeed, by direct computation, we observe that the second and the third derivatives of the symbol $s_\epsilon(\xi)=\frac{\xi^3}{1+\epsilon^2\xi^2}$ are given by $s_\epsilon''(\xi)=\frac{2\xi(3-\epsilon^2\xi^2)}{(1+\epsilon^2\xi^2)^3}$ and $s_\epsilon'''(\xi)=\frac{6(1-6\epsilon^2\xi^2+\epsilon^4\xi^4)}{(1+\epsilon^2\xi^2)^4}$. Hence, it follows that for $\xi\neq0$, 
\begin{equation}\label{eq: derivatives of s}
\left\{\begin{aligned}
s_\ep''(\xi)&=0&&\textup{if and only if}\quad\xi=\pm\frac{\sqrt{3}}{\ep},\\
s_\ep'''(\xi)&=0&&\textup{if and only if}\quad\xi=\pm\frac{\sqrt{3\pm2\sqrt{2}}}{\ep}.
\end{aligned}\right.
\end{equation}
Thus, there are at most three dyadic numbers $N_*$ such that $\pm\frac{\sqrt{3}}{\ep}\in\textup{supp}\eta(\frac{|\cdot|}{N_*})$. Let $\mathcal{B}$ denote the collection of such dyadic numbers.

Suppose that $N\in 2^{\mathbb{Z}}\setminus\mathcal{B}$. Then, by the Bernstein inequality, we obtain  
$$\|S_\ep(t)P_Nu_0\|_{L_t^q(\R; L_x^r)}=\|S_\ep(t)P_{\frac{N}{2}\leq\cdot\leq 2N}P_Nu_0\|_{L_t^q(\R; L_x^r)}\lesssim N^{\frac{1}{q}}\|S_\ep(t)P_Nu_0\|_{L_t^q(\R; L_x^{r_0})},$$
where $\frac{1}{r_0}=\frac{1}{r}+\frac{1}{q}$. Note that $|s_\ep''(\xi)|\gtrsim \tfrac{1}{\ep^4|\xi|^3}\sim\frac{1}{\epsilon^4 N^3}$ in the support of $\eta(\frac{|\cdot|}{N})$ and that $\frac{2}{q}+\frac{1}{r_0}=\frac{1}{2}$. Therefore, by the the standard argument (see \cite{Cazenave2003, KT1998} for instance), one can show that
\begin{equation}\label{eq: KT method}
\|S_\ep(t)P_Nu_0\|_{L_t^q(\R; L_x^r)}
\lesssim N^{\frac{1}{q}}(\epsilon^4N^3)^{\frac{1}{q}}\|u_0\|_{L_x^2}
=(\epsilon N)^{\frac{4}{q}}\|u_0\|_{L_x^2}.
\end{equation}

It remains to show \eqref{eq: dyadic piece Strichartz} for a dyadic number $N_*\in\mathcal{B}$. In this case, the support of the frequency cut-off $\eta(\frac{|\cdot|}{N_*})$ may include frequencies $\xi_*=\pm\frac{\sqrt{3}}{\ep}$ where the second derivative of $s_\epsilon$ vanishes but the third derivative is non-zero (see \eqref{eq: derivatives of s}). Note also that $\textup{supp}\eta(\frac{|\cdot|}{N_*})$ may contains the frequencies $\pm\frac{\sqrt{3\pm2\sqrt{2}}}{\ep}$ where the third derivative of $s_\epsilon'''$ is zero. If this is the case, we further decompose
$$P_{N_*}=P_{N_* ; 1}+P_{N_* ; 2},$$
where $P_{N_* ; 1}$ is defined by $\widehat{P_{N_* ; 1}u}(\xi):=\mathds{1}_{\{|\xi-\xi_*|\leq \frac{1}{100\ep}\}}\hat{u}(\xi)$ and $P_{N_* ; 2}:=P_{N_*}-P_{N_* ; 1}$. Note that if $|\xi-\xi_*|\leq \frac{1}{100\ep}$, then $|s_\ep'''(\xi)|=\frac{6|1-6\ep^2\xi^2+\ep^4\xi^4|}{|1+\ep^2\xi^2|^4}\sim 1$ holds. Therefore, by the standard argument, it follows that 
$$\|S_\ep(t) P_{N_*;1} u_0\|_{L_t^q(\mathbb{R}; L_x^r)}
\sim(\ep N_*)^{\frac{4}{q}}\|S_\ep(t) P_{N_*;1} u_0\|_{L_t^q(\mathbb{R}; L_x^r)} 
\lesssim (\ep N_*)^{\frac{4}{q}} \|u_0\|_{L_x^2},$$
where in the last inequality, we employed the standard method in the proof of \eqref{eq: KT method} with $|s_\ep'''(\xi)|\sim 1$. On the other hand, repeating the proof of  \eqref{eq: KT method}, one can show that 
$$\|S_\ep(t)P_{N_* ; 2}u_0\|_{L_t^q(\R; L_x^r)}\lesssim (\epsilon N_*)^{\frac{4}{q}}\|u_0\|_{L_x^2}.$$
Therefore, \eqref{eq: dyadic piece Strichartz} follows for $N_*\in\mathcal{B}$. 
\end{proof}

\begin{proof}[Proof of Lemma \ref{Lem:high-high BL}]
Since $\frac{\xi\langle \xi\rangle}{1+\epsilon^2\xi^2}\lesssim\frac{1}{\epsilon^2}$, we have 
\begin{equation}\label{eq: high-high BL proof}
\bigg\| \frac{\pa_x}{\langle\epsilon\partial_x\rangle^2}\big((P_{>\frac{1}{5\epsilon}}u)(P_{>\frac{1}{5\epsilon}}v)\big)\bigg\|_{X_{\ep}^{s, -\frac{1-9\delta}{2}}}
\ls \epsilon^{-2} \big\|(P_{>\frac{1}{5\epsilon}}u)(P_{>\frac{1}{5\epsilon}}v)\big\|_{X_{\ep}^{s-1, -\frac{1-9\delta}{2}}}.
\end{equation}
For the right hand side, we note that by the transference principle (Lemma \ref{Lem:linear estimates of Xsb} $(5)$) and the Strichartz estimate (Lemma~\ref{Lem:Strichartz}) with the admissible pair $(q,r)=(\frac{18}{1+2a}, \frac{3}{1-a})$ for sufficiently small $0<a=\frac{5\delta}{1-9\delta}\ll1$, 
$$\|F\|_{L_t^{\frac{18}{1+2a}}L_x^{\frac{3}{1-a}}}\lesssim \epsilon^{\frac{2(1+2a)}{9}}\big\||\partial_x|^{\frac{2(1+2a)}{9}}F\big\|_{X_{\ep}^{0, \frac{1+\delta}{2}}},$$
whose dual inequality is given by 
$$\|F\|_{X_{\ep}^{0, -\frac{1+\delta}{2}}}\lesssim \epsilon^{\frac{2(1+2a)}{9}}\big\||\partial_x|^{\frac{2(1+2a)}{9}}F\big\|_{L_t^{\frac{18}{17-2a}}L_x^{\frac{3}{2+a}}}.$$
Hence, interpolating with the trivial bound $\|F\|_{X_{\ep}^{0, 0}}\leq \|F\|_{L_t^2(\mathbb{R}; L_x^2)}$, we obtain 
$$\|F\|_{X_{\ep}^{0, -\frac{1-9\delta}{2}}}\lesssim \epsilon^{\frac{2}{9}}\||\partial_x|^{\frac{2}{9}}F\|_{L_t^{\frac{18(1+2a)}{17+16a}}L_x^{\frac{3}{2}}}.$$
Thus, going back to \eqref{eq: high-high BL proof}, applying this bound and the fractional Leibniz rule, it follows that 
$$\begin{aligned}
&\bigg\| \frac{\pa_x}{\langle\epsilon\partial_x\rangle^2}\big((P_{>\frac{1}{5\epsilon}}u)(P_{>\frac{1}{5\epsilon}}v)\big)\bigg\|_{X_{\ep}^{s, -\frac{1-9\delta}{2}}}\\
&\lesssim \epsilon^{-\frac{16}{9}} \big\||\partial_x|^{\frac{2}{9}}\big((P_{>\frac{1}{5\epsilon}}u)(P_{>\frac{1}{5\epsilon}}v)\big)\big\|_{L_t^{\frac{18(1+2a)}{17+16a}}W_x^{s-1,\frac{3}{2}}}\\
&\lesssim \epsilon^{-\frac{16}{9}}\Big\{\|P_{>\frac{1}{5\epsilon}}u\|_{L_t^{18}W_x^{(s-1)+\frac{2}{9},3}}\|P_{>\frac{1}{5\epsilon}}v\|_{L_t^{18}L_x^3}+\|P_{>\frac{1}{5\epsilon}}u\|_{L_t^{18}L_x^3}\|P_{>\frac{1}{5\epsilon}}v\|_{L_t^{18}W_x^{(s-1)+\frac{2}{9},3}}\Big\}.
\end{aligned}$$
Note that $\|P_{>\frac{1}{5\epsilon}}u\|_{L_t^{18}W_x^{(s-1)+\frac{2}{9},3}}\lesssim \epsilon^{\frac{5}{9}}\|P_{>\frac{1}{5\epsilon}}u\|_{L_t^{18}W_x^{s-\frac29,3}}$ and $\|P_{>\frac{1}{5\epsilon}}u\|_{L_t^{18}L_x^3}\lesssim \epsilon^{\frac79}\|P_{>\frac{1}{5\epsilon}}u\|_{L_t^{18}W_x^{\frac79,3}}$, because $\hat{u}$ and $\hat{v}$ are supported in $|\xi|>{\frac{1}{\ep}}$. Therefore, by the transference principle (Lemma~\ref{Lem:linear estimates of Xsb} $(5)$) and the Strichartz estimate (Lemma~\ref{Lem:Strichartz}) with the admissible pair $(18,3)$, we prove the proposition.
\end{proof}

Next, we show the desired bounds for the low-low and high-low interactions. For them, one can employ the standard argument. 

\begin{lemma}[Bilinear estimate; low-low interaction]\label{Lem:low-low BL}
For $s\geq 0$, we have
$$\bigg\| \frac{\pa_x}{\langle\epsilon\partial_x\rangle^2}\big((P_{\leq\frac{1}{5\epsilon}}u)(P_{\leq\frac{1}{5\epsilon}}v)\big)\bigg\|_{X_{\ep}^{s, -\frac{1-9\delta}{2}}}
\ls \|u\|_{X_{\ep}^{s,\frac{1+\delta}{2}}} \|v\|_{X_{\ep}^{s,\frac{1+\delta}{2}}}.$$
\end{lemma}

\begin{proof}
Repeating the reduction in Section~\ref{sec: reduction}, one can see that it is enough to show that 
\begin{equation}\label{eq: low-low BL proof}
I_\ep(\tau, \xi):=
\sup_{\xi>0,\ \tau\in\mathbb{R}}\frac{2\xi^2}{\la\ep\xi\ra^4\la\tau-s_\ep(\xi)\ra^{1-9\delta}}\int_\frac{\xi}{2}^\infty\frac{\mathds{1}_{|\xi_1|\leq\frac{2}{5\epsilon}}\mathds{1}_{|\xi-\xi_1|\leq\frac{2}{5\epsilon}}}{\la\tau-s_\ep(\xi-\xi_1)-s_\ep(\xi_1)\ra^{1+\delta}}d\xi_1\lesssim1,
\end{equation}
where compared to \eqref{eq: bilinear for BBM}, the additional indicator functions $\mathds{1}_{|\xi_1|\leq\frac{2}{5\epsilon}}$ and $\mathds{1}_{|\xi-\xi_1|\leq\frac{2}{5\epsilon}}$ come from the low frequency projections. 

To estimate the integral in \eqref{eq: low-low BL proof}, we assume that $|\xi-\xi_1|\leq|\xi_1|\leq\frac{2}{5\ep}$ and $0<\frac{\xi}{2}\leq\xi_1<\infty$. Then, $z=z(\xi_1)=s_\epsilon(\xi_1)+s_\epsilon(\xi-\xi_1)$ obeys
\begin{equation}\label{low frequency integral proof 1}
\begin{aligned}
z'(\xi_1)&=\frac{\xi(2\xi_1-\xi)}{\langle\epsilon\xi_1\rangle^2\langle\epsilon(\xi-\xi_1)\rangle^2}\bigg( 
\frac{2}{\langle\epsilon\xi_1\rangle^2}+\frac{2}{\langle\epsilon(\xi-\xi_1)\rangle^2}-1 \bigg) \\
&\geq\frac{\xi(2\xi_1-\xi)}{{\langle 2/5 \rangle^4}}\bigg(\frac{4}{\langle 2/5 \rangle^2}-1\bigg)\sim \xi(2\xi_1-\xi)>0
\end{aligned}
\end{equation}
so that $s_\epsilon(\xi_1)+s_\epsilon(\xi-\xi_1)$ can be substituted by $z$. We note that $z'(\frac{\xi}{2})=0$ by \eqref{low frequency integral proof 1} and applying the Taylor's theorem, there exists $\xi_*\in(\frac{\xi}{2}, \xi_1)$ such that 
\begin{equation}\label{low frequency integral proof 2}
z(\xi_1)=z(\tfrac{\xi}{2}) +\tfrac12z''(\xi_*)(\xi_1-\tfrac{\xi}{2})^2.
\end{equation}
We also have $z''(\xi_*)>0$ since $z(\xi_1)$ is increasing on $[\frac{\xi}{2},\infty)$. Moreover, the direct computation and the mean value theorem yields, for some $\xi_{**}\in(\xi_*-\xi, \xi_*)\subset(-\xi_1, \xi_1)$,
\begin{equation}\label{low frequency integral proof 3}
z''(\xi_*)= s_\ep''(\xi_*) + s_\ep''(\xi-\xi_*)=s_\ep''(\xi_*) - s_\ep''(\xi_*-\xi)=s_\ep'''(\xi_{**})\xi\leq6\xi,
\end{equation}
because $s_\epsilon''(\xi)=\frac{2\xi(3-\ep^2\xi^2)}{\langle \ep\xi\rangle^6}$ is odd and $
s_\epsilon'''(\xi)=\frac{6[(\ep^2\xi^2-3)^2-8]}{\langle \ep\xi\rangle^8}\leq 6$ for $|\xi|\leq|\xi_1|\leq\frac{2}{5\ep}$.
Combining \eqref{low frequency integral proof 2} and \eqref{low frequency integral proof 3}, we obtain
\begin{equation}\label{low frequency integral proof 4}
\frac{1}{2\xi_1-\xi}
=\frac{\sqrt{z''(\xi_*)}}{2\sqrt{2}\sqrt{z(\xi_1)-z(\frac{\xi}{2})}}
\leq \frac{\sqrt{6}|\xi|^{\frac{1}{2}}}{2\sqrt{2}\sqrt{|z-z(\frac{\xi}{2})|}}.
\end{equation}
Now, we apply \eqref{low frequency integral proof 1}, \eqref{low frequency integral proof 4} and basic calculus inequality to $I_{\ep}(\tau,\xi)$. Then, it follows that 
$$\begin{aligned}
|I_{\ep}(\tau,\xi)|
&\lesssim \frac{|\xi|^{\frac{3}{2}}}{\langle\ep\xi\rangle^{4}\la \tau-s_\ep(\xi)\ra^{\left(1-9\delta\right)  }}\int_{\mathbb{R}}\frac{dz}{\la \tau-z\ra^{1+\delta}\sqrt{|z-z(\frac{\xi}{2})|}}\\
&\lesssim\frac{|\xi|^{\frac{3}{2}}}{\langle\ep\xi\rangle^{4}\la \tau-s_\ep(\xi)\ra^{\left(1-9\delta\right)}\la \tau-z(\frac{\xi}{2})\ra^{\frac{1}{2}}}\lesssim\frac{|\xi|^{\frac{3}{2}}}{\langle\ep\xi\rangle^{4}\la z(\frac{\xi}{2})-s_\ep(\xi)\ra^{\frac{1}{2}}}\lesssim1,
\end{aligned}$$
where in the last step, we used the fact that $z(\frac{\xi}{2})-s_\ep(\xi)=-\frac{3}{4}\xi^3\la\ep\xi\ra^{-2}\la\frac{\ep\xi}{2}\ra^{-2}$.
\end{proof}

\begin{lemma}[Bilinear estimate; high-low interaction]\label{Lem:high-low BL}
For $s\geq 1$, we have
$$\bigg\| \frac{\pa_x}{\langle\epsilon\partial_x\rangle^2}\big((P_{>\frac{1}{5\epsilon}}u)(P_{\leq\frac{1}{5\epsilon}}v)\big)\bigg\|_{X_{\ep}^{s, -\frac{1-9\delta}{2}}}
\ls \|u\|_{X_{\ep}^{s,\frac{1+\delta}{2}}} \|v\|_{X_{\ep}^{s,\frac{1+\delta}{2}}},$$
provided that $u\equiv 0$ and $v\equiv 0$ outside the interval $[-1,1]$.
\end{lemma}

\begin{proof}
First of all, we note that repeating the proof of Lemma~\ref{Lem:high-high BL}, one can show that 
$$\bigg\| \frac{\pa_x}{\langle\epsilon\partial_x\rangle^2}\big((P_{>\frac{1}{5\epsilon}}u)(P_{\frac{1}{50\epsilon}<\cdot\leq\frac{1}{5\epsilon}}v)\big)\bigg\|_{X_{\ep}^{s, -\frac{1-9\delta}{2}}}
\ls \|u\|_{X_{\ep}^{s,\frac{1+\delta}{2}}} \|v\|_{X_{\ep}^{s,\frac{1+\delta}{2}}}.$$
Hence, it remains to show that 
$$\bigg\| \frac{\pa_x}{\langle\epsilon\partial_x\rangle^2}\big((P_{>\frac{1}{5\epsilon}}u)(P_{\leq\frac{1}{50\epsilon}}v)\big)\bigg\|_{X_{\ep}^{s, -\frac{1-9\delta}{2}}}
\ls \|u\|_{X_{\ep}^{s,\frac{1+\delta}{2}}} \|v\|_{X_{\ep}^{s,\frac{1+\delta}{2}}},$$
but as in Section~\ref{sec: reduction}, its proof can be reduced to show that 
\begin{equation}\label{eq: high-low BL proof}
I_\ep(\tau, \xi):=
\frac{2\xi^2}{\la\ep\xi\ra^4\la\tau-s_\ep(\xi)\ra^{1-9\delta}}\int_\frac{\xi}{2}^\infty\frac{\mathds{1}_{|\xi_1|>\frac{1}{10\ep}}\mathds{1}_{|\xi-\xi_1|\leq\frac{1}{25\ep}}}{\la\tau-s_\ep(\xi-\xi_1)-s_\ep(\xi_1)\ra^{1+\delta}}d\xi_1
\end{equation}
is uniformly bounded in $\tau$ and $\xi>0$. Note that if $|\xi_1|>\frac{1}{10\ep}$ and $|\xi-\xi_1|\leq\frac{1}{25\ep}$, then $|\xi|\sim|2\xi_1-\xi|\sim |\xi_1|\gtrsim\frac{1}{\ep}$, and thus,  $z(\xi_1):=s_\ep(\xi_1)+s_\ep(\xi-\xi_1)$ satisfies 

$$|z'(\xi_1)|\sim \frac{|\xi||2\xi_1-\xi|}{(\ep\xi_1)^4}\sim\frac{1}{\ep^4\xi^2}.$$
Hence, by substitution $z=z(\xi_1)$, it follows that 
$$|I_\ep(\tau,\xi)|\lesssim \frac{\xi^2}{\la\ep\xi\ra^4}\int_{\R}\frac{\mathds{1}_{|\xi_1|>\frac{1}{10\epsilon}}\mathds{1}_{|\xi-\xi_1|\leq\frac{1}{25\ep}}}{\la \tau -z\ra^{1+\delta}}\left|\frac{1}{z'(\xi_1)}\right|dz\lesssim1.$$
This shows our lemma.
\end{proof}

\begin{proof}[Proof of Proposition~\ref{Lem:bilinear for BBM}]
The proof follows from the above three bounds (Lemma~\ref{Lem:high-high BL}, \ref{Lem:low-low BL} and ~\ref{Lem:high-low BL}).
\end{proof}

\subsection{Bilinear estimate for the difference in the nonlinear terms}

Next, we prove the following modified bilinear estimate that will be used to measure the difference between the nonlinear parts of the BBM and the KdV equations. Indeed, the key point in this modification is that the lower regularity norm $\|\cdot\|_{X_{\ep}^{0,\frac{1+\delta}{2}}}$ can be taken to the difference (see Section \ref{Sec: Estimate for the difference 1}).

\begin{proposition}[Bilinear estimate for the difference in the nonlinear terms]\label{Lem:BL for difference} For sufficiently small $0<\delta\ll 1$, we have
$$\bigg\| \frac{\pa_x}{\langle\epsilon\partial_x\rangle^2}(uv)\bigg\|_{X_{\ep}^{0, -\frac{1-9\delta}{2}}}
\ls \|u\|_{X_{\ep}^{0,\frac{1+\delta}{2}}} \|v\|_{X_{\ep}^{1,\frac{1+\delta}{2}}}.$$
provided that $u\equiv 0$ and $v\equiv 0$ outside the time interval $[-1,1]$.
\end{proposition}

\begin{proof}
Splitting the high-high, low-low and high-low frequency interactions by frequency projections as before, we write 
$$\begin{aligned}
\bigg\| \frac{\pa_x}{\langle\epsilon\partial_x\rangle^2}(uv)\bigg\|_{X_{\ep}^{0, -\frac{1-9\delta}{2}}}&\leq\bigg\| \frac{\pa_x}{\langle\epsilon\partial_x\rangle^2}\Big(u(P_{>\frac{1}{50\epsilon}}v)\Big)\bigg\|_{X_{\ep}^{0, -\frac{1-9\delta}{2}}}\\
&\quad +\bigg\| \frac{\pa_x}{\langle\epsilon\partial_x\rangle^2}\Big((P_{\leq\frac{1}{5\epsilon}}u)(P_{\leq\frac{1}{50\epsilon}}v)\Big)\bigg\|_{X_{\ep}^{0, -\frac{1-9\delta}{2}}}\\
&\quad +\bigg\| \frac{\pa_x}{\langle\epsilon\partial_x\rangle^2}\Big((P_{>\frac{1}{5\epsilon}}u)(P_{\leq\frac{1}{50\epsilon}}v)\Big)\bigg\|_{X_{\ep}^{0, -\frac{1-9\delta}{2}}} \\
&\quad=:\textup{(I)} + \textup{(II)} + \textup{(III)}.
\end{aligned}$$
Then, we will estimate each term on the right-hand side separately. First, by the low-low interaction bilinear estimate (Lemma~\ref{Lem:low-low BL}), we obtain that $\textup{(II)}\ls \|u\|_{X_{\ep}^{0,\frac{1+\delta}{2}}} \|v\|_{X_{\ep}^{1,\frac{1+\delta}{2}}}$. For $\textup{(I)}$, we repeat but slightly modify the argument in the proof of the high-high interaction bilinear estimate (Lemma~\ref{Lem:high-high BL}). Indeed, using the transference principle and the Strichartz estimate (Lemma~\ref{Lem:Strichartz}) with the admissible pair $(\frac{6}{1-a},\frac{2}{a})$ for sufficiently small $0<a\ll\delta$, we observe that
$$\|F\|_{X_\epsilon^{0,-\frac{1+\delta}{2}}}\lesssim \big\||\epsilon\partial_x|^{\frac{2(1-a)}{3}}F\big\|_{L_t^{\frac{6}{5+a}}L_x^{\frac{2}{2-a}}},$$
by duality. Consequently, interpolating with the trivial bound $\|u\|_{X_\epsilon^{0,0}}\leq\|u\|_{L_t^2L_x^2}$ leads to
$$\|F\|_{X_\epsilon^{0,-\frac{1-9\delta}{2}}}\lesssim\big\||\epsilon\partial_x|^{\frac{2(1-2a)}{3}}F\big\|_{L_t^{\frac{6(1-a)}{5-6a-2a^2}}L_x^{\frac{1}{1-a}}}.$$
Hence, it follows that 
$$\textup{(I)}\lesssim\bigg\| \frac{\partial_x|\epsilon\partial_x|^{\frac{2(1-2a)}{3}}}{\langle\epsilon\partial_x\rangle^2}\Big(u(P_{>\frac{1}{50\epsilon}}v)\Big)\bigg\|_{L_t^{\frac{6(1-a)}{5-6a-2a^2}}L_x^{\frac{1}{1-a}}}\lesssim\frac{1}{\epsilon}\big\|u(P_{>\frac{1}{50\epsilon}}v)\big\|_{L_t^{\frac{6(1-a)}{5-6a-2a^2}}L_x^{\frac{1}{1-a}}},$$
where in the second inequality, we used that by the H\"ormander-Mikhlin theorem, the Fourier multiplier $\frac{\epsilon\partial_x|\epsilon\partial_x|^{\frac{2(1-2a)}{3}}}{\langle\epsilon\partial_x\rangle^2}$ is uniformly bounded in the operator norm from $L^{\frac{1}{1-a}}$ to itself. Thus, by embedding of Fourier restriction norm, H\"older's inequality and the Strichartz estimates (Lemma~\ref{Lem:Strichartz}), we obtain
$$\textup{(I)}\lesssim\frac{1}{\epsilon}\|u\|_{C_tL_x^2}\|P_{>\frac{1}{50\epsilon}}v\|_{L_t^{\frac{3}{a}}L_x^{\frac{2}{1-2a}}}\lesssim\frac{1}{\epsilon^{1-\frac{4a}{3}}}\|u\|_{X_{\ep}^{0,\frac{1+\delta}{2}}} \left\||\partial_x|^{\frac{4a}{3}}P_{>\frac{1}{50\epsilon}}v\right\|_{X_{\ep}^{0,\frac{1+\delta}{2}}}.$$
Finally, using that $P_{>\frac{1}{50\epsilon}}v$ is supported in $|\xi|\gtrsim\frac{1}{\epsilon}$ on the Fourier side, we conclude that 
$$\textup{(I)}\ls \|u\|_{X_{\ep}^{0,\frac{1+\delta}{2}}} \|v\|_{X_{\ep}^{1,\frac{1+\delta}{2}}}.$$
It remains to estimate $\textup{(III)}$. For this, we argue as we did for the low-low frequency interaction bilinear estimate (Lemma~\ref{Lem:low-low BL}). Indeed, following the similar argument in Section~\ref{sec: reduction}, we reduce the proof of the desired bound $\textup{(III)}\ls \|u\|_{X_{\ep}^{0,\frac{1+\delta}{2}}} \|v\|_{X_{\ep}^{1,\frac{1+\delta}{2}}}$ to show that 
$$I_\ep(\tau,\xi):=\frac{2\xi^2}{\langle\ep\xi\rangle^4\la \tau-s_\ep(\xi)\ra^{\left(1-9\delta\right)}}\int_\R\frac{\mathds{1}_{|\xi_1|\geq\frac{1}{10\epsilon}}\mathds{1}_{|\xi-\xi_1|\leq\frac{1}{25\epsilon}}}{\la \tau -s_\epsilon(\xi_1)-s_\epsilon(\xi-\xi_1)\ra^{1+\delta}\langle\xi-\xi_1\rangle^2}d\xi_1$$
is bounded uniformly in $\tau$ and $\xi>0$. In this domain of $\xi$, we have $|2\xi_1-\xi|\sim|\xi_1|\geq\frac{1}{10\epsilon}$ and $|\epsilon(\xi-\xi_1)|\leq\frac{1}{25}$.
To handle the integrand of $I_\ep(\tau, \xi)$, we define $z(\xi_1):=s_\epsilon(\xi_1)+s_\epsilon(\xi-\xi_1)$, as before, and compute
$$\begin{aligned}
|z'(\xi_1)|&=\frac{|\xi||2\xi_1-\xi|}{\langle\epsilon\xi_1\rangle^2\langle\epsilon(\xi-\xi_1)\rangle^2}\left(2\left(\frac{1}{\langle\epsilon\xi_1\rangle^2}+\frac{1}{\langle\epsilon(\xi-\xi_1)\rangle^2} \right)-1 \right)\\
&\gtrsim\frac{|\xi||\xi_1|}{\epsilon^2\xi_1^2}\left(\frac{2}{\langle\frac{1}{25}\rangle^2}-1\right)\sim\frac{|\xi|}{\epsilon^2|\xi_1|}\geq\frac{|\xi|}{\epsilon^2(|\xi-\xi_1|+|\xi|)}.
\end{aligned}$$
Thus, substituting $z=z(\xi_1)$ in $I_\ep(\tau,\xi)$, we obtain that 
$$\begin{aligned}
|I_\ep(\tau,\xi)|&\lesssim\frac{\xi^2}{\langle\ep\xi\rangle^4}\int_{-\infty}^\infty\frac{1}{\la \tau -z\ra^{1+\delta}\langle\xi-\xi_1\rangle^2}\frac{\epsilon^2(|\xi-\xi_1|+|\xi|)}{|\xi|}dz\\
&\lesssim\frac{\epsilon|\xi|}{\langle\ep\xi\rangle^4}\int_{-\infty}^\infty\frac{dz}{\la \tau -z\ra^{1+\delta}}+\frac{\epsilon^2\xi^2}{\langle\ep\xi\rangle^4}\int_{-\infty}^\infty\frac{dz}{\la \tau -z\ra^{1+\delta}}\lesssim1,
\end{aligned}$$
which completes the proof of the proposition.
\end{proof}

\section{Uniform bounds for nonlinear solutions}\label{sec: uniform bounds for nonlinear solutions}

As an application of the bilinear estimates in the previous section, we prove the local well-posedness for the rescaled BBM equation with a uniform bound.

\begin{proposition}[Local well-posedness and uniform bound for $\textup{BBM}_\epsilon$]\label{Prop : BBM LWP}
Let $s\geq 1$, and assume that  
$$\sup_{\epsilon\in(0,1]}\|u_{\ep,0}\|_{H^s}\leq R.$$
Then, there exist $T=T(R)>0$, independent of $\ep\in(0,1]$, and a unique solution $u_\ep(t)\in C_t([-T,T]; H^s)$ to the rescaled BBM equation \eqref{integral BBM} with initial data $u_{\ep,0}$ such that
$$\sup_{\epsilon\in(0,1]}\|u_{\ep}\|_{X_{\ep}^{s,\frac{1+\delta}{2}}} \lesssim R$$
for some sufficiently small $\delta>0$.
\end{proposition}

\begin{remark}
The global well-posedness of the BBM \eqref{eq: BBM normalized} with normalized coefficients has been proved in \cite{BBM1972}, and then it is extended later in a lower a lower regularity space $H^s$ for $s\geq0$. However, this well-posedness does not directly imply Proposition \ref{Prop : BBM LWP} for the rescaled model \eqref{eq: BBM}. 

\end{remark}

The above proposition corresponds to the well-posedness result of the following well-known result for the KdV equation.
\begin{theorem}[Local well-posedness for the KdV equation \cite{KPV1993, KPV1996}]\label{thm : KdV LWP}
Let $s>-\frac{3}{4}$. Then, for $w_0\in H^s(\R)$, there exist $T=T(\|w_0\|_{H^s_x})>0$ and a unique solution $w(t)\in C_t([-T,T]:H^s)$ to the KdV equation \eqref{integral KdV} with initial data $w_0$ such that $\|w\|_{X^{s,\frac{1+\delta}{2}}}\lesssim \|w_0\|_{H^s}$, where $\delta>0$ is a sufficiently small number.
\end{theorem}

\begin{proof}[Proof of Proposition~\ref{Prop : BBM LWP}]
We define 
$$\Phi_\ep(u):=\eta_1(t)S_\ep(t) u_{\ep,0}-\eta_1(t)\int_0^tS_\ep(t-t_1)\eta_T(t_1)\frac{\partial_x}{\la\ep\partial_x\ra^2}u_\ep(t_1)^2 dt_1,$$
where $\eta$ is a smooth cut-off given in Section \ref{subsec: notations} and $\eta_T=\eta(\frac{\cdot}{T})$. Then, applying the basic Fourier restriction norm bound (Lemma~\ref{Lem:linear estimates of Xsb}) and the bilinear estimate (Proposition~\ref{Lem:bilinear for BBM}) to the nonlinear term, we obtain 
$$\begin{aligned}
\|\Phi_\ep(u)\|_{X_{\ep}^{s,\frac{1}{2}+\delta}}
   &\lesssim \|u_{\ep,0}\|_{H^s} + \left\|\eta_{T}\frac{\pa_x}{\la\ep\partial_x\ra^2}u^2\right\|_{X_{\ep}^{s,-(\frac{1}{2}-\delta)}}\\ 
   &\lesssim R+ T^{8\delta}\left\|\frac{\pa_x}{\la\ep\partial_x\ra^2}u^2\right\|_{X_{\ep}^{s,-(\frac{1}{2}-9\delta)}}\lesssim R+T^{8\delta} \|u\|_{X_{\ep}^{s,\frac{1}{2}+\delta}}^2.
\end{aligned}$$
In other words, there exists $c>0$, independent of $\ep\in(0,1]$, such that  
$$\|\Phi_\ep(u)\|_{X_{\ep}^{s,\frac{1+\delta}{2}}}\leq cR + cT^{8\delta}\|u\|_{X_{\ep}^{s,\frac{1}{2}+\delta}}^2.$$
Similarly, for the difference, one can show that 
$$\|\Phi_\ep(u_1)-\Phi_\ep(u_2)\|_{X_{\ep}^{s,\frac{1}{2}+\delta}}\leq cT^{8\delta}\Big(\|u_1\|_{X_\ep^{s,\frac{1+\delta}{2}}}+\|u_2\|_{X_\ep^{s,\frac{1+\delta}{2}}}\Big)\|u_1-u_2\|_{X_\ep^{s,\frac{1+\delta}{2}}}.$$
Thus, taking $T=(\frac{1}{16c^2R})^{\frac{1}{8\delta}}$, we prove that $\Phi_\epsilon$ is contractive on $\{u\in X_{\ep}^{s,\frac{1+\delta}{2}} : \| u\|_{X_{\ep}^{s,\frac{1+\delta}{2}}}\leq 2cR \}$. Consequently, the proposition follows from the Banach fixed point theorem.
\end{proof}

\section{Proof of the main results (Theorem \ref{thm: main result 1.1} and \ref{thm: main result 1.2})}\label{sec: proof of the main result}

This section is devoted to the proof of our main results. In Section \ref{sec: proof of the local theorem}, using the linear and the bilinear estimates obtained in previous sections, we show that the BBM regularization is valid for energy class solutions on a short time interval (Theorem \ref{thm: main result 1.1}). In the proof, in order to circumvent the technical difficulty (see Remark \ref{remark: technical difficiulty to measure difference}), we also adopt the strategy in Hong-Yang \cite{HY2024}, where the KdV is derived from the Boussinesq equation in a similar context. Then, in Section \ref{sec: proof of the global theorem}, we extend the time interval of validity arbitrarily by the conservation laws, which completes the proof of Theorem \ref{thm: main result 1.2}.

\subsection{BBM regularization (Proof of Theorem \ref{thm: main result 1.1})}\label{sec: proof of the local theorem}

Throughout this section, we assume \eqref{eq: main result 1 initial data condition} on initial data for some $1\leq s\leq 5$ and $R>0$, and let $u_\epsilon(t)$ (resp., $w(t)$) be the global solution to $\textup{BBM}_\epsilon$ \eqref{eq: BBM} (resp., the KdV \eqref{eq: KdV}) with initial data $u_{\ep,0}$ (resp., $w_0$). Then, we aim to measure the difference $(u_\epsilon-w)$, but some technical issue arises in the use of the Fourier restriction norm.

\begin{remark}\label{remark: technical difficiulty to measure difference}
It is not easy to estimate the difference $(u_\epsilon-w)$ directly in the Fourier restriction norm, because the norm is too sensitive to the choice of the associated linear flow. Indeed, the rescaled linear BBM flow $S_\ep(t)$ behaves completely differently from the Airy flow $S(t)$ in high frequencies; if $|\xi|\gg\frac{1}{\ep}$, then $s_\ep(\xi)$ is almost linear, i.e., $\approx \frac{\xi}{\ep^2}$ but $s(\xi)=\xi^3$ is cubic. Thus, high regularity is required to bound $\|S_\ep(t)u_0\|_{X^{0,b}}$ and $\|S(t)u_0\|_{X_\ep^{0,b}}$, which is not good as our goal is to reduce the regularity requirement.
\end{remark}

To overcome the problem, following the strategy in Hong-Yang \cite{HY2024}, we introduce the following auxiliary equation which we call the frequency localized rescaled BBM equation 
\begin{equation}\label{Auxiliary equation}
v_{\ep}(t)=S_\ep(t)P_{\leq N_0}u_{\ep,0} -\int_0^tS_\ep(t-t_1)  P_{\leq N_0}\frac{\partial_x }{\la\ep\partial_x\ra^2}(P_{\leq N_0}v_\ep(t_1))^2dt_1,
\end{equation}
where $P_{\leq N_0}$ is the Littlewood-Paley projection onto the frequency $|\xi|\leq 2N_0$ with $N_0=\frac{1}{2}\epsilon^{-\frac{2}{5}}$ (see Section \ref{subsec: notations}). For this equation, one can show the local well-posedness and uniform bounds similar to Proposition~\ref{Prop : BBM LWP} repeating the analysis in the previous section, because the equation \eqref{Auxiliary equation} has almost the same structure as the rescaled BBM \eqref{integral BBM} as well as $P_{\leq N_0}$ only makes the terms smaller.

From now on, we fix $1\leq s\leq 5$, let $0<\delta\ll 1$ be sufficiently small, and set
\begin{equation}\label{T choice}
T:=\min\left\{ 1, \left(\frac{c_0}{R}\right)^{(1/8\delta)}\right\}
\end{equation}
where $c_0>0$ is a small number to be chosen later. Then, by the assumptions on the initial data \eqref{eq: main result 1 initial data condition}, it follows from Theorem~\ref{thm : KdV LWP} and \ref{Prop : BBM LWP} that $u_{\ep}(t)$, $v_{\ep}(t)$ and $w(t)$ all exist on the same time interval $[-T,T]$, and they satisfy the uniform bounds
\begin{equation}\label{all solution uniform bounds}
\|u_{\ep}\|_{X_{\ep}^{s,\frac{1}{2}+\delta}}+\|v_{\ep}\|_{X_{\ep}^{s,\frac{1}{2}+\delta}}+\|w\|_{X^{s,\frac{1}{2}+\delta}}\lesssim R.
\end{equation}

In a sequel, we estimate the differences $(u_\epsilon-v_\epsilon)$ and $(v_\epsilon-w_\epsilon)$ separately. Indeed, the two nonlinear solutions $u_\epsilon$ and $v_\epsilon$ share the same linear propagator $S_\ep(t)$, and thus the difference $\|u_\epsilon-v_\epsilon\|_{X_\ep^{0,\frac{1}{2}+\delta}}$ can be measured. On the other hand, the difference $(v_\epsilon-w_\epsilon)$ can be estimated by the norm $\|\cdot\|_{X^{s,\frac{1}{2}+\delta}}$ associated with the Airy flow, since tricky high frequencies is cut off by $P_{N_0}$ in \eqref{Auxiliary equation} so that the norm comparability (Lemma~\ref{Lem:norm equiv in low freq}) can be used.

\subsubsection{Estimate for the difference $\|u_\epsilon-v_\epsilon\|_{X_\ep^{0,\frac{1}{2}+\delta}}$}\label{Sec: Estimate for the difference 1}

Inserting the temporal smooth cut-offs in the integral equations \eqref{integral BBM} and \eqref{Auxiliary equation}, we write the difference
$$(u_{\ep}-v_{\ep})(t)=\eta_1(t)S_\ep(t)(1-P_{\leq N_0})u_{\ep,0}- \eta_1(t)\int_0^tS_\ep(t-t_1)\eta_T(t_1)\frac{\pa_x}{\langle\epsilon\partial_x\rangle^2}(u_\ep^2-v_\ep^2)(t_1)dt_1,$$
where $\eta_T=\eta(\frac{\cdot}{T})$. Then, by the linear estimates in the Fourier restriction norms (Lemma~\ref{Lem:linear estimates of Xsb}), we have
$$\begin{aligned}
\|u_{\ep}-v_{\ep}\|_{X_{\epsilon}^{0,\frac{1}{2}+\delta}}
&\lesssim \|(1-P_{\leq N_0})u_{\ep,0}\|_{L^2}+ \left\|\eta_{T}\frac{\pa_x}{\la\ep\partial_x\ra^2}\big((u_\ep+v_\ep)(u_\ep-v_\ep)\big)\right\|_{X_{\ep}^{0,\delta-\frac{1}{2}}}\\
&\lesssim N_0^{-s}\|u_{\ep,0}\|_{H^{s}}+ T^{8\delta}\bigg\|\frac{\pa_x}{\la\ep\partial_x\ra^2}\big((u_\ep+v_\ep)(u_\ep-v_\ep)\big)\bigg\|_{X_{\ep}^{0,-(\frac{1}{2}-9\delta)}}.
\end{aligned}$$
The bilinear estimate for the difference estimate (Proposition~\ref{Lem:BL for difference}) and the uniform bounds \eqref{all solution uniform bounds} with ${N_0}=\frac{1}{2}\epsilon^{-\frac{2}{5}}$ yield
$$\begin{aligned}
\|u_{\ep}-v_{\ep}\|_{X_{\epsilon}^{0,\frac{1}{2}+\delta}}
&\lesssim N_0^{-s}R
+ T^{8\delta}\big(\|u_{\ep}\|_{X_{\epsilon}^{1,\frac{1}{2}+\delta}}+\|v_{\ep}\|_{X_{\epsilon}^{1,\frac{1}{2}+\delta}}\big)\|u_{\ep}-v_{\ep}\|_{X_{\epsilon}^{0,\frac{1}{2}+\delta}}\\
&\lesssim \epsilon^{\frac{2}{5}s}R+ \frac{c_0}{R}\cdot 2R\|u_{\ep}-v_{\ep}\|_{X_{\epsilon}^{0,\frac{1}{2}+\delta}}
\lesssim \ep^{\frac{2}{5}s}R+ 2c_0\|u_{\ep}-v_{\ep}\|_{X_{\epsilon}^{0,\frac{1}{2}+\delta}}.
\end{aligned}$$
Therefore, taking sufficiently small $c_0>0$, it follows that $\|u_{\ep}-v_{\ep}\|_{X_{\ep}^{0,\frac{1}{2}+\delta}}\lesssim_R\ep^{\frac{2}{5}s}$.

\subsubsection{Estimate for the difference $\|v_{\ep}-w\|_{X^{0,\frac{1}{2}+\delta}}$}
Similarly as before, we put the temporal cut-offs in the difference between the integral equations for $v_{\ep}$ and $w$, and take the Fourier restriction norm associated with the Airy flow. Subsequently, we obtain that
$$\begin{aligned}
\|v_{\ep}-w\|_{X^{0,\frac{1}{2}+\delta}}
&\leq \|\eta_1(t)(P_{\leq N_0}-1)w\|_{X^{0,\frac{1}{2}+\delta}}
+\| \eta_1(t)(S_\ep(t)-S(t))P_{\leq N_0}u_{\ep,0}\|_{X^{0,\frac{1}{2}+\delta}}\\
&\quad +\|\eta_1(t)S(t)P_{\leq N_0}(u_{\ep,0}-w_0)\|_{X^{0,\frac{1}{2}+\delta}} \\
&\quad +\left\|\eta_1(t)\int_0^tS_\ep(t-t_1)\eta_{T}(t_1)\frac{\pa_x}{\la\ep\partial_x\ra^2}P_{\leq N_0}\big((v_\ep+w)(v_\ep-w)(t_1)\big)dt_1\right\|_{X^{0,\frac{1}{2}+\delta}}\\
&\quad + \left\|\eta_1(t)\int_0^t(S_\ep(t-t_1)-S(t-t_1))\eta_{T}(t_1)\frac{\pa_x}{\la\ep\partial_x\ra^2}P_{\leq N_0}(w(t_1))^2dt_1\right\|_{X^{0,\frac{1}{2}+\delta}}\\
&\quad + \left\|\eta_1(t)\int_0^tS(t-t_1)\eta_{T}(t_1)\partial_x\left( 1 - \frac{1 }{\la\ep\partial_x\ra^2} \right)P_{\leq N_0}(w(t_1))^2dt_1\right\|_{X^{0,\frac{1}{2}+\delta}} \\
& \quad =:\textup{(I)} + \textup{(II)} + \textup{(III)} + \textup{(IV)} + \textup{(V)} + \textup{(VI)}.
\end{aligned}$$
In the upper bound, we group the terms that can be estimated using similar methods, and handle them together.
First, applying the linear estimates in the Fourier restriction norms (Lemma~\ref{Lem:linear estimates of Xsb}), we obtain
\begin{equation}\label{ineq : v-w 1}
\begin{aligned}
\textup{(I)} + \textup{(III)} 
& \lesssim \|(1-P_{\leq N_0})w\|_{X^{0,\frac{1}{2}+\delta}} 
+\|u_{\ep,0}-w_0\|_{L_x^2} \\
& \lesssim { \ep^{\frac{2}{5}s}\|w\|_{X^{s,\frac{1}{2}+\delta}} }
+\|u_{\ep,0}-w_0\|_{L_x^2} 
\lesssim \ep^{\frac{2}{5}s}R 
+\|u_{\ep,0}-w_0\|_{L_x^2}.
\end{aligned}
\end{equation}
For the second term, we observe that for the low frequency $|\xi|\leq N_0$, by Taylor's theorem,
\begin{equation}\label{eq: propagator symbol difference bound}
|e^{\frac{it\xi^3}{1+\epsilon^2\xi^2}}-e^{it\xi^3}|\leq \min\left\{2, |t|\left|\frac{\xi^3}{1+\epsilon^2\xi^2}-\xi^3\right|\right\}\lesssim\min\left\{1, |t|\epsilon^2|\xi|^5\right\}.
\end{equation}
Therefore, it follows that 
\begin{equation}\label{ineq : v-w 1.1}
\textup{(II)}
\lesssim \|(e^{\frac{it\xi^3}{1+\epsilon^2\xi^2}}-e^{it\xi^3})\mathds{1}_{\{|\xi|\leq N_0\}} \hat{u}_{\ep,0}\|_{L_\xi^2(\mathbb{R})}
\lesssim
\ep^{\frac{2}{5}s}\|u_{\ep,0}\|_{H^s} 
\lesssim \ep^{\frac{2}{5}s}R.
\end{equation}

For $\textup{(V)}$, using Lemma~\ref{Lem:difference linear sol in Xsb}, the bilinear estimate for the linear KdV equation (See Remark~\ref{cf. KdV flow}) and the choice of $T$, we have
\begin{equation}\label{ineq : v-w 2}
\textup{(V)}
\lesssim \ep^{\frac{2}{5}s}T^{1-\delta}\left\|\frac{\partial_x}{\la\ep\partial_x\ra^2}w^2\right\|_{X^{s,\delta-\frac{1}{2}}}
\lesssim \ep^{\frac{2}{5}s}T^{8\delta}\|w\|_{X^{0,\frac{1}{2}+\delta}}^2
\lesssim c_0\ep^{\frac{2}{5}s}R,
\end{equation}
where in the last inequality, we used that $0<T\leq 1$ and $1-\delta < 8\delta$ for sufficiently small $\delta>0$.
For the remaining terms, applying the inhomogeneous term estimate in the Fourier restriction norm (Lemma~\ref{Lem:linear estimates of Xsb} (4)) and the fact that $1-\frac{1}{\la\ep\xi\ra^2}=\frac{\ep^2\xi^2}{\la\ep\xi \ra^2}\lesssim \ep^{\frac{2}{5}s}|\xi|^{\frac{2}{5}s}\lesssim \ep^{\frac{2}{5}s}\la \xi\ra^{s}$ for $|\xi|\leq \frac{1}{2}\ep^{-\frac{2}{5}}=N_0$, we obtain
$$\textup{(IV)} + \textup{(VI)}
 \lesssim  \left\|\eta_T(t_1)\frac{\partial_x}{\la\ep\partial_x\ra^2}P_{\leq N_0}(v_\ep^2-w^2)\right\|_{X^{0,\delta-\frac{1}{2}}}
+\ep^{\frac{2}{5}s}\left\|\la\partial_x\ra^s\eta_T(t_1)\partial_x(w)^2\right\|_{X^{0,\delta-\frac{1}{2}}}.$$
Then, the linear estimates in the Fourier restriction norms (Lemma~\ref{Lem:linear estimates of Xsb}), the bilinear estimate for the linear KdV equation (see Remark~\ref{cf. KdV flow}) lead to
\begin{equation}\label{ineq : v-w 3}
\begin{aligned}
\textup{(IV)} + \textup{(VI)}
& \lesssim T^{\frac{1}{4}-\delta}\|P_{\leq N_0}\partial_x(v_\ep^2-w^2)\|_{X^{0,-\frac{1}{4}}} 
+ T^{\frac{1}{4}-\delta}\ep^{\frac{2}{5}s}\|\partial_x(w)^2\|_{X^{s,-\frac{1}{4}}}\\
& \lesssim T^{\frac{1}{4}-\delta}(\|P_{\leq N_0}v_\ep\|_{X^{0,\frac{1}{2}+\delta}}+\|w\|_{X^{0,\frac{1}{2}+\delta}})\|v_\ep-w\|_{X^{0,\frac{1}{2}+\delta}}\\ 
&\quad+ T^{\frac{1}{4}-\delta}\ep^{\frac{2}{5}s}{ \|w\|^2_{X^{s,\frac{1}{2}+\delta}} }\\
& \lesssim 2T^{8\delta}R\|v_\ep-w\|_{X^{0,\frac{1}{2}+\delta}} 
+ T^{8\delta}\ep^{\frac{2}{5}s}R^2\\
& \lesssim 2c_0\|v_\ep-w\|_{X^{0,\frac{1}{2}+\delta}}
+ c_0\ep^{\frac{2}{5}s}R.
\end{aligned}
\end{equation}
In the above, for $P_{\leq N_0}v_\ep$, we used that $\|P_{\leq N_0}v_\ep\|_{X^{0,\frac{1}{2}+\delta}} \sim \|P_{\leq N_0}v_\ep\|_{X_\ep^{0,\frac{1}{2}+\delta}}\lesssim R$ holds by the comparability of norm in low frequencies (Lemma~\ref{Lem:norm equiv in low freq}).

Finally, summing up \eqref{ineq : v-w 1}, \eqref{ineq : v-w 1.1}, \eqref{ineq : v-w 2} and \eqref{ineq : v-w 3}, we prove that $$\|v_{\ep}-w\|_{X^{0,\frac{1}{2}+\delta}}\lesssim \ep^{\frac{2}{5}s}R 
+\|u_{\ep,0}-w_0\|_{L_x^2}+c_0\ep^{\frac{2}{5}s}R+2c_0\|v_\ep-w\|_{X^{0,\frac{1}{2}+\delta}},$$
which immediately implies 
$$\|v_{\ep}-w\|_{X^{0,\frac{1}{2}+\delta}}\lesssim_R \ep^{\frac{2s}{5}}+\|u_{\ep,0}-w_0\|_{L^2_x},$$
since $c_0>0$ is sufficiently small.

\subsection{Extension of the interval of validity (Proof of Theorem \ref{thm: main result 1.2})}\label{sec: proof of the global theorem}
We will extend the interval of validity for the BBM regularization using the conservation laws in \eqref{eq: conservation laws}. Here, we consider the positive time direction only, because the negative one can be treated by the same way. Suppose that
$$\sup_{\epsilon\in(0,1]}\|u_{\ep,0}\|_{H^1}, \|w_0\|_{H^1}\leq R,$$
and let $u_\epsilon(t)$ (resp., $w(t)$) be the global solution to $\textup{BBM}_\epsilon$ \eqref{eq: BBM} (resp., the KdV \eqref{eq: KdV}) with initial data $u_{\ep,0}$ (resp., $w_0$).

For $\textup{BBM}_\epsilon$, by the conservation law $E_{\textup{BBM}_\epsilon}^{(1)}[u_\epsilon(t)]=E_{\textup{BBM}_\epsilon}^{(1)}[u_{\epsilon,0}]$, it follows that 
$$\|u_\epsilon(t)\|_{L^2}^2\leq E_{\textup{BBM}_\epsilon}^{(1)}[u_{\epsilon}(t)]=E_{\textup{BBM}_\epsilon}^{(1)}[u_{\epsilon,0}]\leq R^2.$$
On the other hand, for the conservation law $E_{\textup{BBM}_\epsilon}^{(2)}[u_\epsilon(t)]=E_{\textup{BBM}_\epsilon}^{(2)}[u_{\epsilon,0}]$, applying the Gagliardo-Nirenberg inequality
$$\|u\|_{L^3}^3\leq C_{GN}\|u\|_{L^2}^{\frac{5}{2}}\|\partial_x u\|_{L^2}^{\frac{1}{2}}$$
to the cubic terms in $E_{\textup{BBM}_\epsilon}^{(2)}[u_\epsilon(t)]$ and Young's inequality $ab\leq \frac{3}{4}a^{\frac{4}{3}}+\frac{1}{4}b^4$, we obtain 
$$\begin{aligned}
\frac{R^2}{2}+\frac{1}{3}C_{GN}R^3
& \geq E_{\textup{BBM}_\epsilon}^{(2)}[u_{\epsilon,0}]
=E_{\textup{BBM}_\epsilon}^{(2)}[u_\epsilon(t)]\\
&\geq\frac{1}{2}\|\partial_x u_\epsilon(t)\|_{L^2}^2
-\frac{C_{GN}}{3}\|u_\epsilon(t)\|_{L^2}^{\frac{5}{2}}\|\partial_x u_\epsilon(t)\|_{L^2}^{\frac{1}{2}}\\
&\geq \frac{1}{4}\|\partial_x u_\epsilon(t)\|_{L^2}^2-\frac{3}{4}\bigg(\frac{C_{GN}}{3}\bigg)^{\frac{4}{3}}\|u_\epsilon(t)\|_{L^2}^{\frac{10}{3}}\\
&\geq \frac{1}{4}\|\partial_x u_\epsilon(t)\|_{L^2}^2-\frac{3}{4}\bigg(\frac{C_{GN}}{3}\bigg)^{\frac{4}{3}}R^{\frac{10}{3}}
\end{aligned}$$
so that $\|\partial_x u_\epsilon(t)\|_{L^2}^2\leq 2R^2+3(\frac{C_{GN}}{3})^{\frac{4}{3}}R^{\frac{10}{3}}$. Repeating the same estimate for the KdV \eqref{eq: KdV}, we prove the same bound. Therefore, we prove that there exists $L>0$ such that
\begin{equation}\label{eq: global-in-time energy bound}
\sup_{\epsilon\in(0,1]}\|u_{\epsilon}(t)\|_{C(\mathbb{R}; H^1)}, \|w(t)\|_{C(\mathbb{R}; H^1)}\leq L\max\{R, R^\frac{5}{3}\}.
\end{equation}

Now, we choose a short time $T=T(L\max\{R, R^\frac{5}{3}\})>0$ from Theorem \ref{thm: main result 1.1}. Then, for each $k\in\mathbb{N}$, applying Theorem \ref{thm: main result 1.1} with \eqref{eq: global-in-time energy bound} to the evolutions from time $(k-1)T$ to $kT$, it follows that 
$$\|u_{\ep}(t)-w(t)\|_{C_t([(k-1)T, kT]; L_x^2)}\leq C_R\|u_{\ep}((k-1)T)-w((k-1)T)\|_{L^2_x}+C_R\ep^{\frac{2}{5}},$$
for some $C_R\geq 2$. Here, an important remark is that $T$ and $C_R$ depend only on $L\max\{R, R^{\frac{5}{3}}\}$, but independent of $\epsilon\in(0,1]$ and $k\in\mathbb{N}$. Hence, by simple induction, we obtain that 
$$\begin{aligned}
\|u_{\ep}(t)-w(t)\|_{C_t([0, kT]; L_x^2)}&\leq (C_R)^k\|u_{\ep,0}-w_0\|_{L^2_x}+\ep^{\frac{2}{5}}\big\{C_R+(C_R)^2+(C_R)^3+\cdots (C_R)^k\big\}\\
&\sim (C_R)^k\big\{\|u_{\ep,0}-w_0\|_{L^2_x}+\ep^{\frac{2}{5}}\big\}.
\end{aligned}$$
Therefore, we conclude that there exists $K>0$ such that 
$$\|u_{\ep}(t)-w(t)\|_{L_x^2}\lesssim \big\{\|u_{\ep,0}-w_0\|_{L^2_x}+\ep^{\frac{2}{5}}\big\} e^{K|t|},$$
because for any large $|t|\geq 1$, there exists $k$ such that $|t|\sim k T$.

\appendix

\section{Proof of conservation laws for the rescaled BBM equation}\label{sec: proof of conservation laws for the rescaled BBM equation}

For readers' convenience, we present formal calculations to show the three conservation laws \eqref{eq: conservation laws} for the rescaled BBM equation \eqref{eq: BBM}. Then, a rigorous proof follows from the standard persistence of regularity argument (see \cite{Cazenave2003}, for example). 

Suppose that $u_\epsilon=u_\epsilon(t)$ is a smooth and localized solution to
$$\partial_t(1-\epsilon^2\partial_x^2)u_\epsilon+\partial_x^3u_\epsilon+\partial_x(u_\epsilon^2)=0.$$
Then, by direct calculations and substituting $\partial_t u_\epsilon$ by the equation, we prove the first two conservation laws;
$$\frac{d}{dt}E_{\textup{BBM}_\epsilon}^{(0)}[u_\epsilon]=\int_{\mathbb{R}} \partial_tu_\epsilon dx=\int_{\mathbb{R}} \partial_x\big(\epsilon^2\partial_t\partial_x u_\epsilon-\partial_x^2 u_\epsilon-u_\epsilon^2\big)dx=0$$
and
$$\begin{aligned}
\frac{d}{dt}E_{\textup{BBM}_\epsilon}^{(1)}[u_\epsilon]&=2\int_{\mathbb{R}} u_\epsilon\partial_t(1-\epsilon^2\partial_x^2)u_\epsilon dx=-2\int_{\mathbb{R}} u_\epsilon\big\{\partial_x^3u_\epsilon+\partial_x(u_\epsilon^2)\big\}dx\\
&=\int_{\mathbb{R}} \partial_x\bigg\{(\partial_x u_\epsilon)^2-\frac{4}{3}u_\epsilon^3\bigg\}dx=0.
\end{aligned}$$
For the third quantity, using the equation again and substituting $V=\frac{1}{1-\epsilon^2\partial_x^2}(\partial_x^2u_\epsilon+u_\epsilon^2)$, we prove that  
$$\begin{aligned}
\frac{d}{dt}E_{\textup{BBM}_\epsilon}^{(2)}[u_\epsilon]&=-\int_{\mathbb{R}} \big(\partial_x^2u_\epsilon+u_\epsilon^2\big)\partial_tu_\epsilon dx=\int_{\mathbb{R}} \big(\partial_x^2u_\epsilon+u_\epsilon^2\big) \bigg\{\frac{\partial_x}{1-\epsilon^2\partial_x^2}\big(\partial_x^2u_\epsilon+u_\epsilon^2\big)\bigg\}dx\\
&=\int_{\mathbb{R}} \big\{(1-\epsilon^2\partial_x^2)V\big\} (\partial_xV)dx=\frac{1}{2}\int_{\mathbb{R}} \partial_x\big\{V^2+\epsilon^2(\partial_xV)^2\big\}dx=0.
\end{aligned}$$

\bibliographystyle{abbrv}
\bibliography{Reference}

\end{document}